\documentclass[a4paper,10pt]{amsart}
\usepackage[utf8]{inputenc}
\usepackage{amsfonts, amsmath, amssymb}
\usepackage[left=3cm,right=3cm]{geometry}
\usepackage{amsthm}
\usepackage{graphicx}
\usepackage{url}
\usepackage{subfigure}

\newtheorem{tma}{Theorem}[section]
\newtheorem{prop}{Proposition}

\newtheorem{lema}{Lemma}

\def\grad{\mathop{\rm grad}\nolimits}

\def\Hess{\mathop{\rm Hess}\nolimits}

\def\Ric{\mathop{\rm Ric}\nolimits}

\title{An asymptotically cusped three dimensional expanding gradient Ricci soliton}

\author[Daniel Ramos]{Daniel Ramos \\ \\ \tiny{January 10, 2013}}
\begin{document}

\begin{abstract}
We construct an expanding gradient Ricci soliton in dimension three over the topological manifold $\mathbb R \times \mathbb T^2  $ that
aproaches asymptotically a constant curvature cusp at one end, and a flat manifold on the other end.
We prove that this is
the only gradient soliton with this topology, provided the curvature is negatively pinched, $-1/4 < sec < 0$, at the
time-zero manifold
(normalizing
the
soliton to be born at time $-1$).
\end{abstract}

\thanks{Universitat Autònoma de Barcelona, Departament de Matemàtiques - 80193 Bellaterra, Barcelona (Spain)- \\ E-mail:
\texttt{dramos@mat.uab.cat} . }

%\footnotetext{Draft, \today}
\maketitle

\section{Introduction}
A gradient Ricci soliton is a smooth riemannian metric $g$ on a manifold $M$ together with a potential function $f:M\rightarrow \mathbb R$
such that
\begin{equation}
\Ric + \Hess f +\frac{\epsilon}{2}g=0 \label{solitoneq} 
\end{equation}
for some $\epsilon\in \mathbb R$. Solitons provide special examples of self-similar solutions of Ricci flow, $\frac{\partial}{\partial t} g
= -2 \Ric$, evolving by homotheties and diffeomorphisms generated by the flow $\phi(\cdot,t)$ of the vector field $\grad f$, this is
$g(t)=(\epsilon t +1)\phi^*_t g_0$. The constant $\epsilon$ can be normalized to be $-1,0,1$ according the soliton being shrinking, steady
or expandig respectively (see \cite{TRFTAA1} for a general reference). Solitons play an important role in the
classification of singular models for Ricci flow despite of (or actually due to) existing only a limited number of examples. In dimension
3, the only closed gradient solitons are those of constant curvature. Furthermore, by the results of Hamiton-Ivey \cite{Ham1}, \cite{Ivey}
and Perelman \cite{Per}, the only
three-dimensional open gradient shrinking
solitons with bounded curvature are $\mathbb S^3$, $\mathbb R^3$ and $\mathbb S^2 \times \mathbb R$ with their standard metrics, and their
quotients. Notice that in all these examples the gradient vector field is null. Examples with nontrivial
potential function in dimension three include the Gaussian flat soliton, the steady Bryant soliton, the product of the
2-dimensional steady Cigar soliton with $\mathbb R$ due to Hamilton, and a continuous family of rotationally symmetric expanding gradient
solitons due to Bryant. In summary, shrinking and steady solitons are very few, and these are useful in the analysis of high curvature
regions of the Ricci flow. Expanding
solitons are less understood and there is much more variety of them.

\vspace{1em}
A couple of motivating examples are the following. The hyperbolic metric $g=dr^2 + e^{-2r}(dx^2+dy^2)$ on $\mathbb R^3$ together with a
trivial potential $f=cst$ fits into the soliton
equation (\ref{solitoneq}) with $\epsilon=1$, so any quotient (hyperbolic manifold) yields an expanding soliton. An open quotient of
the hyperbolic space, for instance the cusp $\mathbb R \times \mathbb T^2$ obtained as a quotient by parabolic isometries (represented by
euclidean translations on the $xy$-plane) yields also a soliton. 

We will restrict ourselves to bounded curvature metrics that yield uniformly bounded curvature flows. From the PDE point of view, this
condition ensures existence and uniqueness of solutions of the flow, both in the compact case (\cite{Ham2}, \cite{DeTur}) as well as on the
noncompact one (\cite{Shi}, \cite{ChenZhu}). If the uniformly boundedness condition of the curvature is dropped, we can loose the
uniqueness;
for instance approximating a cusp by high-curvature capped ends (cf. with \cite{Top1} and \cite{Ram} in the 2-dimensional case).
Nevertheless, even with this assumption, it is not clear what an open expanding hyperbolic manifold is at the birth time, namely
$t\rightarrow -1$ when the evolution is $g(t)=(t+1)g_{\mathrm{hyp}}$. A sequence of shrinked negatively curved manifolds does not need
to have a limit in the Gromov-Hausdorff sense, since the curvature is not uniformly bounded below. A hyperbolic cusp on $\mathbb R \times
\mathbb T^2$ as an expanding soliton tends to a line while its curvature tends to $-\infty$, as $t\rightarrow -1$ the birth time. 

Another interesting phenomenon occurs in $\mathbb R^n$ endowed with the euclidean metric: it fits into the soliton equation together with a
potential function $f(p)=-\epsilon \frac{|p|^2}{4}$ for any $\epsilon\in\mathbb R$. The nonzero cases are the so called Gaussian solitons,
and even when the metric is constant in all cases (hence there is a unique solution with a given initial condition), there is more than one
soliton structure on it.

The aim of this paper is constructing a particular example of expanding gradient Ricci soliton on $\mathbb R \times
\mathbb T^2$, different from the constant curvature examples. Furthermore, we prove that it is the only possible nonhomogeneous soliton
on this manifold provided there is a lower sectional curvature bound equal to $-\frac{1}{4}$.

\vspace{1em}
In Section \ref{secexist} we consider the generic metric $g=dr^2 + e^{2h} (dx^2 + dy^2)$ over $M=\mathbb R \times \mathbb T^2$, where
$h=h(r)$ is
a function determining the size of the foliating flat tori, and a potential function $f=f(r)$ constant over these tori. We find a suitable
choice of $h$ and $f$ that makes the triple $(M,g,f)$ a soliton solution for the Ricci flow with bounded curvature, by means of the phase
portrait analisys of the soliton ODEs.

\begin{tma} \label{tma_exist}
There exists an expanding gradient Ricci soliton $(M,g,f)$ over the topological manifold $M = \mathbb R \times \mathbb T^2$ satisfying the
following properties:
\begin{enumerate}
 \item The metric has pinched sectional curvature $-\frac{1}{4} < sec < 0$.
 \item The soliton approaches the hyperbolic cusp expanding soliton on one end.
 \item The soliton approaches locally the flat Gaussian expanding soliton on a cone on the other end.
\end{enumerate}
More preciselly, $M$ admits a metric
$$g=dr^2 + e^{2h(r)} (dx^2 + dy^2)$$
where $(r,x,y)\in \mathbb R \times \mathbb S^1 \times \mathbb S^1$; and a potential function $f=f(r)$, satisfying the soliton equation and
with the stated bounds on the curvature, such that
$$h\sim \frac{r}{2} \quad \mathrm{and} \quad f \rightarrow cst \qquad  \mathrm{as} \quad r\rightarrow -\infty$$
and
$$h\sim \ln r \quad \mathrm{and} \quad f \sim -\frac{r^2}{4} \qquad  \mathrm{as} \quad r\rightarrow +\infty .$$
\end{tma}
For the asymptotical notation ``$\sim$'', we write
$$\phi(r) \sim \psi(r) \quad \mathrm{as} \quad r\rightarrow \infty$$
if$$\lim_{r\rightarrow \infty} \frac{\phi(r)}{\psi(r)} = 1.$$

Let us remark that when $r\rightarrow +\infty$ the theorem states that the metric approaches $g=dr^2 + r^2 (dx^2 + dy^2)$. This is a
nonflat cone over the torus, namely its curvatures are $sec_{rx}=sec_{ry}=0$ and $sec_{xy}=-\frac{1}{r^2}$, but it indeed
approaches a flat metric when $r\rightarrow +\infty$.

This example is interesting in regard of the following known fact.
\begin{prop} \label{propcao}
Let $(M^n, g,f)$ with $n\geq 3$ be a complete noncompact gradient expanding soliton with $\Ric\geq \frac{-1}{2} + \delta$ for some
$\delta>0$. Then $f$ is a strictly concave exhaustion function, that achieves one maximum, and the underlying manifold $M^n$ is
diffeomorphic to $\mathbb R^n$. 
\end{prop}

Cf. with \cite{Caoetal}, Lemma 5.5 and Remark 5.6. Our example proves that the lemma fails if one only assumes $\Ric>\frac{-1}{2}$. In
this critical case the soliton also has a strictly concave potential, but in this case $f$ has no maximum and it is not exhausting, and
actually this solution admits a different topology for the manifold, namely $\mathbb R \times \mathbb T^2$.

\vspace{1em}
In Section \ref{secuniq} we consider the general case of a metric over $\mathbb R \times \mathbb T^2$ with $sec>-\frac{1}{4}$, and we
prove that the only nonflat solution is the example previously constructed. The lower bound on the curvature implies a concavity property
for the potential function. This leads together with the prescribed topology to a general form of the coordinate expression of the metric,
that can be subsequently computed as the example. 

\begin{tma} \label{tma_uniq}
Let $(M,g,f)$ be a nonflat gradient Ricci soliton over the topological manifold $M = \mathbb R \times \mathbb T^2$ with bounded curvature
$sec > -\frac{1}{4}$. Then it is the expanding gradient soliton depicted on Theorem (\ref{tma_exist}).
\end{tma}

\vspace{1em}
In section \ref{seccurv} we explore a growth property of the scalar curvature on our soliton. In nonnegative sectional curvature, the
evolution of $Rm$ and Harnack inequalities \cite{Ham3} imply that $\frac{d R}{dt}\geq - \frac{R}{t}$ pointwise, and in particular $\frac{d
R}{dt}\geq 0$ for all $t$ if the solution is also ancient. Our example exposes that this is not the
case in negative scalar curvature, even with a soliton solution. Despite the self-similarity, the behaviour of the curvature growth is
different
at different times. The combined effects of the diffeomorphism translation and the homothety act in opposite manner. For short time after
birth, the negative curvature is increasing everywhere. After some small time, there appear points where the curvature is
decreasing, but eventually all points recover the increase of the scalar curvature and the limit of the curvature is zero for
every
fixed point as $t\rightarrow +\infty$.

\begin{tma} \label{tma_evolcur}
Let $(M,g(t))$ be the (soliton) Ricci flow defined on $M=\mathbb R \times \mathbb T^2 $ and for $t\in(-1,+\infty)$, such that $g(0)=g_0$
where $g_0$ is the metric constructed in Theorem \ref{tma_exist}. Let $R=R(t)$ be the scalar curvature of $g(t)$. Then there exists
$\delta<0$ such that
\begin{itemize}
\item for all $t\in (\delta, +\infty)$ there exist points in $M$ with $\frac{\partial R}{\partial t} >0$ and points with $\frac{\partial
R}{\partial t} <0$
\item for some $-1<t<\delta$, it is satisfied $\frac{\partial R}{\partial t} >0$ everywhere in $M$.
\end{itemize}
\end{tma}

\vspace{1em}
Most tedious computations thorough the paper can be performed and checked using Maple or other similar software, therefore no step-by-step
computations will be shown. Pictures were drawn using Maple and the P4 program \cite{progP4}.

\emph{Acknowledgements:} The author was partially supported by Feder/Mineco through the Grant MTM2009-0759. He is indebt to his advisor,
Joan Porti, for all his guidance. He also wishes to thank Joan Torregrosa for pointing out the technique of compactified phase portraits.

\section{The asymptotically cusped soliton} \label{secexist}
Let us consider the metric
\begin{equation}
g=dr^2 + e^{2h(r)} (dx^2 + dy^2) \label{metrica} 
\end{equation}
where $h=h(r)$ is a one-variable real function, and a potential function $f=f(r)$ depending also only on the $r$-coordinate. The underlying topological manifold can be taken $(r,x,y)\in 
\mathbb R \times \mathbb S^1\times \mathbb S^1 = \mathbb R \times \mathbb T^2$ since $\mathbb R^3$ with this metric admits the appropriate
quotient on the $x,y$
variables.
Standard riemannian computations yield the following equalities.

\begin{lema}\label{lema_qtties}
The metric in the form (\ref{metrica}) associates the following geometric quantities:
\begin{align*} 
 \Ric &= -2((h')^2 + h'')dr^2 -e^{2h}(h''+2(h')^2)(dx^2 + dy^2) ,\\
 R &= -4 h'' -6 (h')^2 ,\\
 sec_{xy} &= -(h')^2  ,\\
 sec_{rx}=sec_{ry} &= -((h')^2 + h'') ,\\
 \Hess f &= f'' dr^2 + e^{2h}f'h'(dx^2 + dy^2) ,\\
 \Delta f &= 2h'f' + f'' ,\\
 \grad f &= f' \textstyle \frac{\partial}{\partial r} ,\\
 \nabla f &= f' dr ,\\
 |\grad f|^2 &= (f')^2 .\\
\end{align*}
\end{lema}

Hence the soliton equation (\ref{solitoneq}) for this metric turns into
$$\left( \frac{\epsilon}{2} + f'' -2(h')^2 -2 h''\right) dr^2 + e^{2h}\left( \frac{\epsilon}{2}+h'f'-2(h')^2 - h''\right) (dx^2 + dy^2)
=0  .$$
This tensor equation is equivalent to the ODEs system
\begin{equation}
 \left\{ \begin{array}{rcl}
	  \frac{\epsilon}{2} + f'' -2(h')^2 -2 h'' &=& 0 \\
	  \frac{\epsilon}{2}+h'f'-2(h')^2 - h''&=& 0 .
         \end{array} \label{eqssolcusp}
\right.\end{equation}
Let us remark that this system would be of second-order in most coordinate systems, but in ours we can just change variables $H=h'$ and
$F=f'$, and rearrange to get a first-order system
\begin{equation*}
\left\{ \begin{array}{rcl}
          H' & = & HF -2 H^2 + \frac{\epsilon}{2} \\
          F' & = & 2HF -2 H^2 + \frac{\epsilon}{2} .	 
         \end{array}
\right.
\end{equation*}

We can solve qualitativelly this system using a phase portrait analysis (see Figure \ref{retrato}). Every trajectory on the phase portrait
represents a soliton, but will not have in general bounded curvature. Actually, bounded curvature is achieved if and only if both $H$
and
$H'$ are bounded on the trajectory.

\begin{figure}[ht]
\centering
\includegraphics[width=0.6\textwidth]{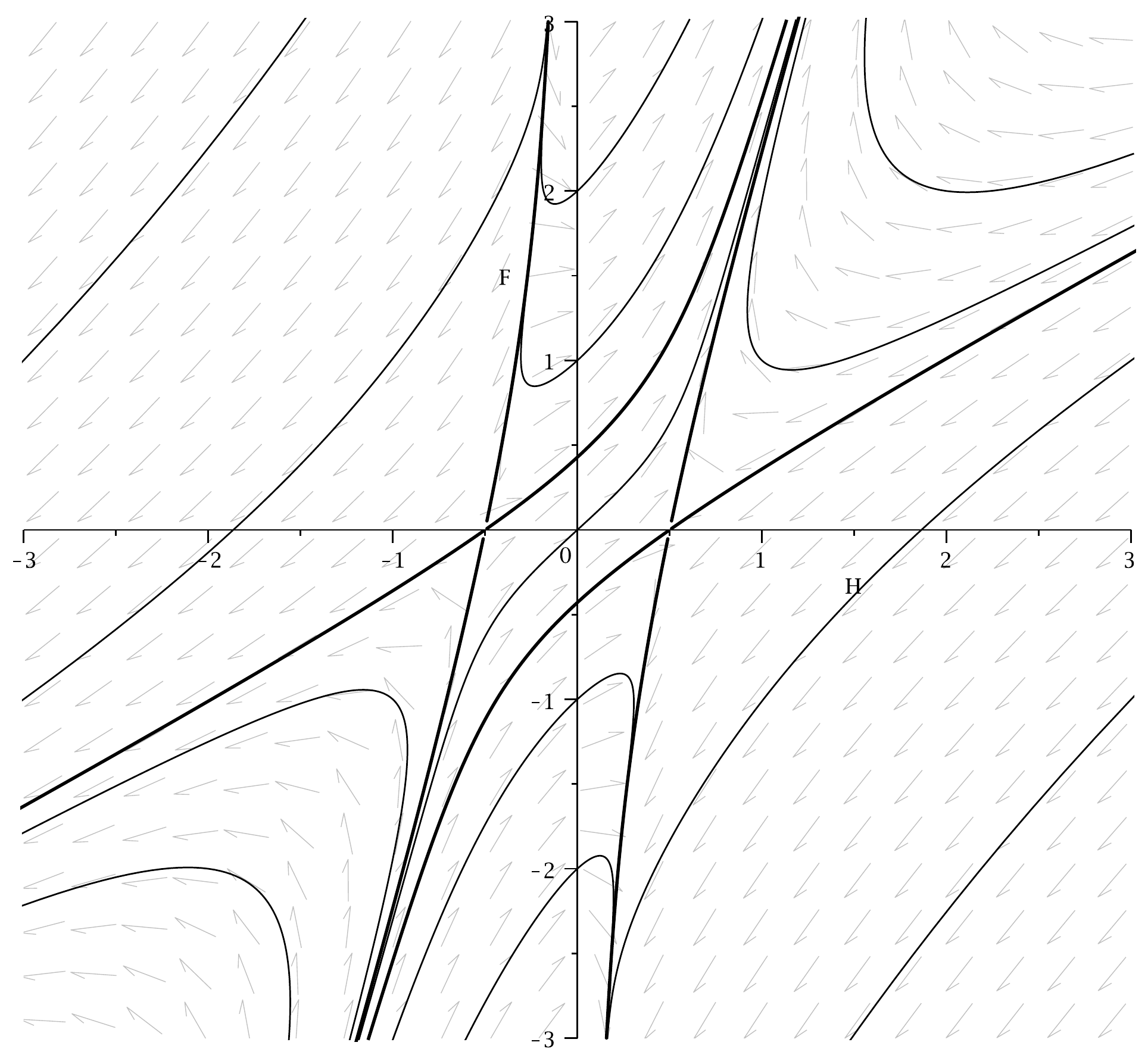}
\caption{Phase portrait of the system (\ref{eqsphase}).}
\label{retrato}
\end{figure}

The critical points (stationary solutions) of the system are found by solving $\{ H'=0 , F'=0 \}$. If the soliton is shrinking
($\epsilon=-1$), there are no critical points and no trajectories with bounded curvature, agreeing with Perelman's classification. If the
soliton is steady ($\epsilon=0$), there is a whole straight line $\{H=0\}$ of fixed
points representing all of them the flat steady soliton. In this case there are neither trajectories with bounded $H$, hence all
solutions have unbounded negative curvature at least in one end. Let us assume henceforth that the soliton is
expanding ($\epsilon=1$), so our system is
\begin{equation}
\left\{ \begin{array}{rcl}
          H' & = & HF -2 H^2 + \frac{1}{2} \\
          F' & = & 2HF -2 H^2 + \frac{1}{2} .	 
         \end{array} \label{eqsphase}
\right.
\end{equation}
There are two critical points, $$(H,F)=(\pm\frac{1}{2},0).$$
The critical point $(\frac{1}{2},0)$ corresponds to a soliton with $h(r)=\frac{r}{2}+c_1$ and $f(r)=c_2$, the gradient
vector field is null, and the metric is $g_0=dr^2 + e^{r+c_1}(dx^2+dy^2)$, which is a complete hyperbolic metric, with constant
sectional curvature equal to $- \frac{1}{4}$, and possesses a cusp at $r\rightarrow -\infty$. As a Ricci flow it is $g(t)=(t+1)g_0$, it
evolves only by homotheties, and it is born at $t=-1$. The symmetric critical point $(-\frac{1}{2},0)$ represents the
same soliton, just reparameterizing $r\rightarrow -r$.

The phase portrait of the system (\ref{eqsphase}) has a central symmetry, that is, the whole phase portrait is invariant
under the change $(H,F,r)\rightarrow (-H,-F,-r)$, so it is enough to analyze one critical point and half the trajectories.

We shall see that the critical points are saddle points, and there is a separatrix trajectory emanaing from each one of them that
represents the soliton metric we are looking for. Both trajectories represent actually the same soliton up to reparameterization.

\begin{lema}
 Besides the stationary solutions, and up to the central symmetry, there is only one trajectory $S$ with bounded $H$. This trajectory is a
separatrix joining a critical point and a point in the infinity on a vertical asymptote.
\end{lema}

\begin{proof}
The linearization of the system is 
$$\left(\begin{array}{c}
        H' \\ F'
        \end{array} \right)
= \left(\begin{array}{cc}
	F-4H & H \\
	2F-4H & 2H
        \end{array} \right)
 \left(\begin{array}{c}
        H \\ F
        \end{array} \right)
.$$
The matrix of the linearized system has determinant $-4H^2 \leq 0$, so the critical points are saddle points. For each one, there are
two eigenvectors determining four separatrix trajectories; being two of them attractive, two of them repulsive, according to the sign of the
eigenvalue. 

We are interested in one of the two repulsive separatrix emanating from the critical point $(H,F)=(\frac{1}{2},0)$, pointing towards the
region $\{H<\frac{1}{2}, F<0\}$. We shall see that has this is the only solution curve (together with its symmetrical) with bounded $H$
along its trajectory, so it represents a metric with bounded curvature.

In order to analyze the asymptotic behaviour of the trajectories, we perform a projective compactification of the plane, as
explained for instance in \cite{DumLliArt}, Ch. 5. The compactified plane maps into a disc where pairs of antipodal points on the boundary
represent
the asymptotic directions, Figure \ref{retrato_compact} shows the compactified phase portrait of (\ref{eqsphase}). 
\begin{figure}[ht]
\centering
\includegraphics[width=0.4\textwidth]{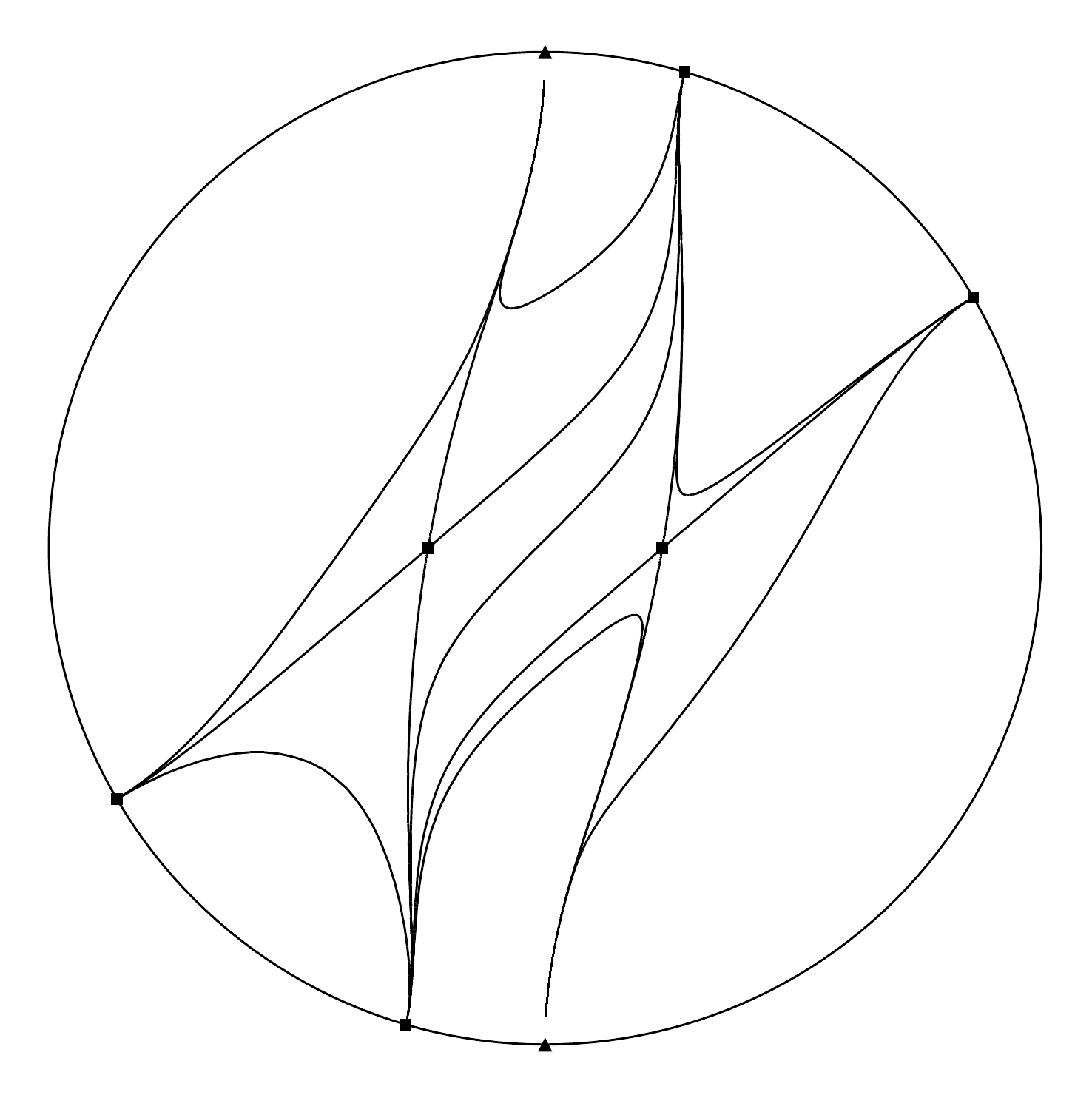}
\caption{Compactified phase portrait of the system (\ref{eqsphase}).}
\label{retrato_compact}
\end{figure}
A standard technique for polynomial systems is to perform a change of charts on the projective plane so that critical points at infinity can
be studied. A sketch is as follows: a polynomial system 
$$\left\{\begin{array}{l}\dot x = P(x,y) \\ \dot y = Q(x,y)  \end{array} \right. $$ 
can be thought as lying on the $\{z=1\}$ plane in the $xyz$-space. By a central projection this maps to a vector field and a phase portrait
on the unit sphere, or in the projective plane after antipodal identification. In order to do this, it may be necessary to resize
the vector field as 
$$\left\{\begin{array}{l}\dot x = \rho(x,y) P(x,y) \\ \dot y = \rho(x,y) Q(x,y)  \end{array} \right. $$
so that the vector field keeps bounded norm on the equator. However, this change only reparameterizes the trajectories. A global picture can
be obtained by orthographic projection of the sphere on the equatorial disc, as in Figure \ref{retrato_compact}, or it can be projected
further to a plane $\{x=1\}$ or $\{y=1\}$ in order to study the critical points at the infinity.
Let us remark that this technique works only for polynomial systems since the polynomial growth ratio suits the algebraic change of
variables.

In our system, this analysis yields that for every trajectory the ratio $H/F$ tends to either $0$, $1+\frac{\sqrt{2}}{2} $ or
$1-\frac{\sqrt{2}}{2} $ as $r\rightarrow \pm \infty$;
represented by the pairs of antipodal critical points (of type node) at infinity. The knowledge of the finite and infinite critical points,
together with their type, determines qualitativelly the phase portrait of Figure \ref{retrato_compact} by the Poincaré-Bendixon theorem.
Thus, a trajectory with bounded $H$ on the $\mathbb R^2$ portrait, when seen on the $\mathbb RP^2$ portrait must have their ends either on
the finite saddle points or on the infinity node with $H/F$ ratio equal to $0$ (meaning a vertical asymptote). The only trajectory
satisfying this condition is the claimed separatrix and its symmetrical. 
\end{proof}

We shall see that this trajectory $S$ is parameterized by $r\in(-\infty,+\infty)$,
and when $r\rightarrow -\infty$ the function $h(r)$ behaves as $\frac{r}{2}$ and then the solution is asymptotically a cusp. Similarly, we
will see that $h', h'' \rightarrow 0$ when $r\rightarrow +\infty$ and then the solution is asymptotically flat.

To better understand the phase portrait it is useful to consider some isoclinic lines. This will give us the limit values for $H$,
$H'$ and the range of the parameter.
\begin{lema} \label{lemaisocl}
The vertical asymptote for the trajectory $S$ occurs at $H=0$. Furthermore, it is parameterized by $r\in(-\infty,+\infty)$ and
$H,H'\rightarrow 0$ as $r\rightarrow +\infty$.
\end{lema}

\begin{proof}

The vertical isocline $\{H'=0\}$ is the
hyperbola
$$F=2H-\frac{1}{2H}$$
and the trajectories cross it with vertical tangent vector. The horizontal isocline $\{F'=0\}$ is the hyperbola
$$F=\frac{1}{2}\left( 2H-\frac{1}{2H}\right)$$
and the trajectories cross it with horizontal tangent vector (see Figure \ref{retrato_closeup}). An oblique isocline is the hyperbola
$$F=2\left( 2H-\frac{1}{2H}\right) ,$$
since over this curve the vector field has constant direction:
$$(H',F')\big|_{(H,4H-\frac{1}{H})} = \left(2H^2-\frac{1}{2},6H^2-\frac{3}{2}\right) =
\left(2H^2-\frac{1}{2}\right) (1,3) .$$

\begin{figure}[ht]
\centering
\includegraphics[width=0.5\textwidth]{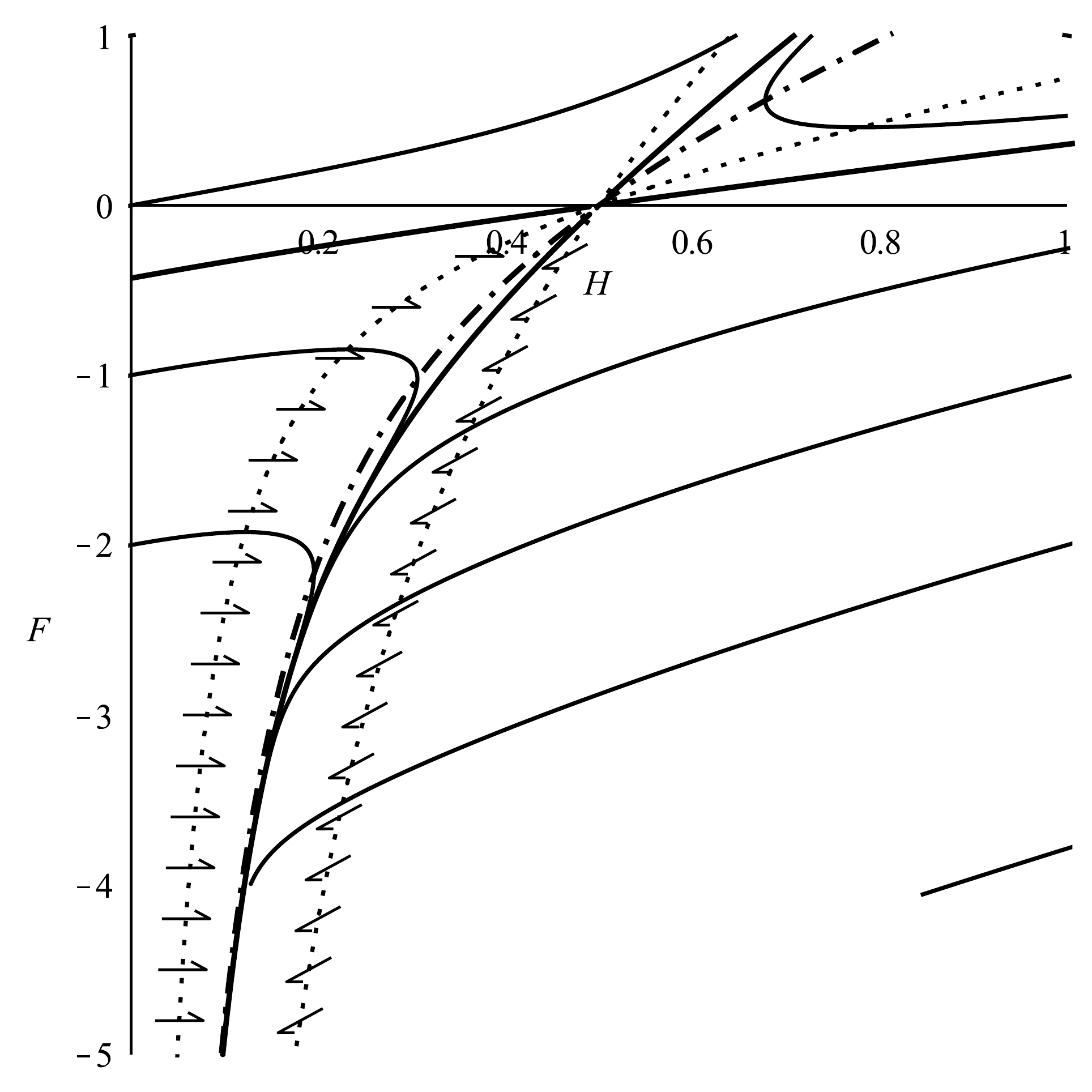}
\caption{Close-up of the separatrix trajectories (bold lines); the vertical isocline (dashes and dots); and the horizontal and oblique
isoclines (dots and arrows).}
\label{retrato_closeup}
\end{figure}

All three isoclines intersect at the critical points. Furthermore, the tangent directions at the critical point $(\frac{1}{2},0)$
have
slope $\frac{dF}{dH}\big |_{H=\frac{1}{2}} = 4$ for the vertical isocline, $\frac{dF}{dH}\big |_{H=\frac{1}{2}} = 2$ for the horizontal
one, and $\frac{dF}{dH}\big |_{H=\frac{1}{2}} = 8$ for the oblique one. The separatrix lines emanating from the critical point follow the
directions given by the eigenvectors of the matrix
of the linearized system (evaluated at the point), that is, the matrix

$$ \left(\begin{array}{cc}
	-2 & \frac{1}{2} \\
	-2 & 1
        \end{array} \right)
$$
whose eigenvalues are $\frac{-1+\sqrt{5}}{2}$ and $\frac{-1-\sqrt{5}}{2}$ with eigenvectors
$$ \left(\begin{array}{c}
        1 \\ 3+\sqrt{5}
        \end{array} \right)
\quad , \quad  
   \left(\begin{array}{c}
        1 \\ 3-\sqrt{5}
        \end{array} \right)
$$
respectivelly, so the separatrix lines have tangent directions with slope $3+\sqrt{5}\simeq  5.24$ and $3-\sqrt{5} \simeq 0.76$
respectively.
The repulsive separatrices are the ones associated with the positive eigenvalue, that is $\frac{-1+\sqrt{5}}{2}$, and the slope is $5.24$.

The repulsive separatrix emanating towards $\{H<\frac{1}{2} , F<0 \}$ initially lies below both the vertical and horizontal
isoclines, so it moves downwards and leftwards; and above the oblique one. The horizontal and oblique isoclines form two barriers for the
separatrix, this is, the separatrix cannot cross any of them. This is obvious for the horizontal one, since the flow is rightwards and the
trajectory is on the right. For the oblique isocline, we just check that any generic point on the isocline $(H,4H-\frac{1}{H})$ has
tangent vector $(1,4+\frac{1}{H^2})$ and a normal vector $\nu = (-4-\frac{1}{H^2},1)$ pointing leftwards and upwards for $0<H<\frac{1}{2}$.
The scalar product of the normal vector $\nu$ and the vector field $(H',F')$ over this isocline is
$$\langle \nu, (H',F')\rangle = \left(2H^2-\frac{1}{2}\right) \left(-4-\frac{1}{H^2}+3\right)=-2H^2+\frac{1}{2H^2}-\frac{3}{2}>0$$
whenever $0<H<\frac{1}{2}$. This means that the flow is always pointing to the left-hand side of the isocline branch and therefore is a
barrier. This proves that the separatrix moves downwards between the two barriers and therefore $H\rightarrow 0$. 

Actually, the vertical isocline is also a barrier for the separatrix. For, if at some point it touched the vertical
isocline, it would then move vertically downwards, keeping the trajectory on the right-hand side of the isocline. There would be then a
tangency, but it is impossible since the tangent vector should be vertical. Since the vertical isocline is becoming itself vertical, this
means that the vertical isocline acts as an atractor for the trajectories. Indeed, the trajectory lies initially in the region
$\{H>0,H'<0\}$ but $H$ must remain positive since it cannot cross the vertical isocline. Therefore $H$ is positive and decreasing, so $H'$
must tend to $0$. This implies that the trajectory tends to the vertical isocline. It is important to
note that both the vertical and horizontal isoclines come close together when $H\rightarrow 0$, but the trajectories stick to the vertical
one much faster than to the horizontal one.

We now see that $r\in(-\infty,+\infty)$. This follows inmediately from the Hartman-Grobman theorem for the case $r\rightarrow -\infty$, but
the trajectory might, a priori, escape to infinity in finite time. This would require that the velocity tangent vector tends to infinity in
finite time, but this is impossible, since $H$ and $H'$ are bounded, thus $F'$ is bounded and hence the tangent vector $(H',F')$ is bounded.
\end{proof}

We have seen that not only $H\rightarrow 0$ as $r\rightarrow +\infty$, but also that $H' \rightarrow 0$. That is,
$h',h''\rightarrow 0$ as $r\rightarrow +\infty$, so all the sectional, scalar and Ricci curvatures tend to zero, the metric becoming
asymptotically flat.

At this point we have seen the existence of the soliton asserted in Theorem \ref{tma_exist}. We now give some more detailed information
about the
asymptotic behaviour of $f$ and $h$ at the ends of the manifold.

\begin{lema} \label{lemaasymp}
The asymptotic behaviour of $f$ and $h$ is
$$ h \sim \frac{r}{2} \quad \mathrm{and} \quad f \rightarrow cst \qquad  \mathrm{as} \quad r\rightarrow -\infty$$
and
$$h\sim \ln r \quad \mathrm{and} \quad f \sim -\frac{r^2}{4} \qquad  \mathrm{as} \quad r\rightarrow +\infty .$$
\end{lema}

\begin{proof}
Recall a version of the l'Hôpital rule: if $$\lim_{x\rightarrow \infty} \phi(x) = \lim_{x\rightarrow \infty} \psi(x) = 0, \pm \infty \quad
\mathrm{and} \quad \lim_{x\rightarrow \infty} \frac{\phi'(x)}{\psi'(x)} = c$$ then $$\lim_{x\rightarrow \infty} \frac{\phi(x)}{\psi(x)} = c.$$

The case when $r\rightarrow -\infty$ follows from Hartman-Grobman theorem: the phase portrait in a small neighbourhood of a saddle critical
point has a flow that is Hölder conjugate to the flow of a standard linear saddle point 
$$\left\{ \begin{array}{l}
   \dot x = x \\
   \dot y = -y ,
  \end{array}\right.
$$
whith solution $x(t) = k_1 e^t$, $y(t)=k_2 e^{-t}$.
This means that $r$ is defined from $-\infty$ onwards, that $H\rightarrow
\frac{1}{2}$ and $F,H',F'\rightarrow 0$ and $\frac{H-1/2}{F} \rightarrow 3+\sqrt{5}$ as $r\rightarrow -\infty$. Since
$H=h'\rightarrow \frac{1}{2}$ then $h\rightarrow -\infty$ and $\frac{H}{1/2}\rightarrow 1$. By l'Hôpital,
$\frac{h}{r/2}\rightarrow 1$.
Using more accurately the Hartman-Grobman theorem, there exists a Hölder function $\eta:U\subset \mathbb R \rightarrow \mathbb R$ defined on
a neighbourhood of zero, and constants $\alpha, C >0$ such that $F(r)=\eta(k_1 e^{r})$ and 
$$ | F(r) - F(r_0) | = |\eta(k_1 e^{r}) - \eta(k_1 e^{r_0}) | \leq C |k_1 e^{r} - k_1 e^{r_0} |^\alpha .$$
When $r_0 \rightarrow -\infty$, we obtain
$$ |F(r)| \leq \tilde C e^{\alpha r} ,$$
thus, $F$ is integrable on an interval $(-\infty,c]$ and thus $f\rightarrow f_0=cst$ as $r\rightarrow -\infty$. The constant $f_0$ is
actually $\sup f$ and can be chosen since it bears no geometric meaning. A more accurate description of $f$ using l'Hôpital tells
$$\lim_{r\rightarrow -\infty} \frac{h-r/2}{f-f_0}=3+\sqrt{5} .$$

For the case when $r \rightarrow +\infty$, we know from the trajectories that $H,H' \rightarrow 0$. Letting $r\rightarrow + \infty$ in the
first equation of (\ref{eqsphase}) we deduce that $HF\rightarrow -\frac{1}{2}$. Using this and letting $r\rightarrow + \infty$ in the second
equation of (\ref{eqsphase}) we conclude that $F'\rightarrow -\frac{1}{2}$, that is, $F' \sim -\frac{1}{2}$ and by l'Hôpital, $F\sim
-\frac{r}{2}$ and $f\sim -\frac{r^2}{4}$ as $r\rightarrow +\infty$.

Now, since $\lim_{r\rightarrow +\infty} \frac{F}{r} = - \frac{1}{2}$, we have
$$-\frac{1}{2} = \lim_{r\rightarrow +\infty} HF = \lim_{r\rightarrow +\infty} \frac{H}{r^{-1}} \frac{F}{r} = -\frac{1}{2} \lim_{r\rightarrow +\infty} \frac{H}{r^{-1}}$$
thus $H\sim \frac{1}{r}$ and therefore $h\sim \ln r$ as $r\rightarrow +\infty$.
\end{proof}

It remains only to check the bounds on the sectional curvatures.

\begin{lema}
The metric (\ref{metrica}) with the function $f$ obtained as solution of the system (\ref{eqsphase}) has bounded sectional curvature
$$-\frac{1}{4}<sec<0 .$$
\end{lema}
\begin{proof}
The expression for the sectional curvatures is given in Lemma \ref{lema_qtties}. The case 
$$sec_{xy}=-(h')^2 = -H^2$$
is trivial since $0<H<\frac{1}{2}$ and therefore $-\frac{1}{4}<sec_{xy}<0$, tending to $-\frac{1}{4}$ on the cusp end, and to $0$ on the wide end. The other sectional curvatures are
\begin{align*}
sec_{rx}=sec_{ry} & = -((h')^2+h'') \\
		  & = -H^2 - H' \\
		  & = H^2 - HF -\frac{1}{2} \\
		  & = -\frac{1}{2}\left( F'+\frac{1}{2} \right)
\end{align*}
We saw in Lemma \ref{lemaisocl} that $\{F'=0\}=\{2HF-2H^2+\frac{1}{2}=0\}$ is a barrier for the separatrix $S$. Hence, $F'<0$ along $S$
and therefore $sec_{rx}, sec_{ry}>-\frac{1}{4}$. Similarly, the set $\{H^2-HF-\frac{1}{2}=0\}$ can also be checked to be a barrier for
$S$ (actually a barrier on the opposite side), and hence $sec_{rx}, sec_{ry}<0$. We also saw in Lemma \ref{lemaasymp} the asymptotics
of $F'$, therefore we have $-\frac{1}{2}<F'<0$ and $sec_{rx}=sec_{ry}$ tend to $-\frac{1}{4}$ on the cusp end, and to $0$ on the wide end.
\end{proof}
This finishes the description of the soliton stated on Theorem \ref{tma_exist}.

\section{Uniqueness}\label{secuniq}
Let $(M,g,f)$ be a gradient expanding Ricci soliton over $\mathbb R \times \mathbb T^2$ such that $ sec > -1/4$. Then
$$\Ric + \Hess f + \frac{1}{2} g =0 ,$$
$$sec > 1/4,$$
$$\Ric > -1/2,$$
$$R > -3/2 .$$
Recall a basic lemma about solitons, that can be proven just derivating, contracting and commuting covariant derivatives on the soliton
equation, see \cite{TRFTAA1}.
\begin{lema}
It is satisfied
$$ R + \Delta f + 3/2 =0 ,$$
$$ g(\grad R , \cdot ) = 2 \Ric(\grad f , \cdot) ,$$
$$ R + |\grad f |^2 + f = C .$$
\end{lema}

Since the soliton is defined in terms of the gradient of $f$, we can arbitrarily add a constant to $f$ without effect. We use this to set
$C=-3/2$ above so that we have
$$ \Delta f = |\grad f|^2 + f .$$
The bound on the curvature implies
$$\Hess f < 0 ,$$
$$\Delta f < 0 ,$$
$$ \langle \grad R , \grad f \rangle > -|\grad f|^2 .$$
First equation means that $f$ is a strictly concave function ($-f$ is a strictly convex function), i.e. $-f\circ \gamma$ is a strictly
convex real function for every (unit speed) geodesic
$\gamma$. This is a strong condition, since then the superlevel sets $A_c=\{f \geq c\}$ are totally convex sets, i.e. every geodesic
segment joining two points on $A_c$ lies entirely on $A_c$. Second equation is just a weaker convexity condition. This concavity on
this topology implies that $f$ has no maximum.

\begin{lema}
The function $f$ is negative and has no maximum.
\end{lema}
\begin{proof}
Note that $f$ is bounded above since $f=\Delta f -|\grad f|^2 < 0 $. 
Now suppose by contradiction that the maximum of $f$ is attained at some point of $\mathbb R \times \mathbb T^2$, then we can
lift this point, the metric and the potential function to the universal cover $ \mathbb R \times \mathbb R^2$. There is then a lattice of
points in the cover where the lifted function $\tilde f$ attains its maximum. But this is impossible since a strictly concave function
cannot have more than one maximum (the function restricted to a geodesic segment joining two maxima would not be strictly concave).

\end{proof}

\emph{Remark.} As stated in Proposition \ref{propcao}, if $\Ric > -\frac{1}{2} + \delta$ for any $\delta>0$, then $f$ has a maximum, and the
set
$A=\{f>f_{max} -\mu\}$ is compact and homeomorphic to a ball for small $\mu$. The function $f$ is then an exhaustion function, this is, the
whole manifold retracts onto $A$ via the flowline of $f$ and therefore $M\cong \mathbb R^3$. Thus this stronger bound on the curvature is
not
compatible with $M\cong \mathbb R \times \mathbb T^2$.

Now we prove that level sets of $f$ are compact.
\begin{lema}
The function $f$ is not bounded below and the level sets $\{f=c\}$ are compact.
\end{lema}
\begin{proof}
Consider $M=\mathbb T^2\times \mathbb R$ as splitted into $\mathbb T^2 \times (-\infty, 0] \cup \mathbb T^2\times [0,+\infty)$, each
component containing one of the two ends. Since $f$ has no maximum, there is a sequence of points $\{x_i\}$ tending to one end such that
$f(x_i)\rightarrow \sup f$. Let us assume that this end is $\mathbb T^2 \times (-\infty,0]$. Then when approaching the opposite end $f$ is
unbounded. Indeed, suppose by contradiction that there is a sequence of points $\{y_i\}$ tending to the $+\infty$ end such that
$f(y_i)\rightarrow -K
>-\infty$. There is a minimizing geodesic segment $\gamma_i$ joining $x_i$ with $y_i$. This gives us a sequence of geodesic paths (whose
length tends to infinity), each one crossing the central torus $\mathbb T^2\times \{0\}$. Since both the torus and the
space of directions of a point are compact, there is a converging subsequence of crossing points together with direction vectors that
determine a sequence of geodesic segments with limit a geodesic line $\gamma$. Now we look
at $f$ restricted to $\gamma$, this is $f\circ \gamma : \mathbb R \rightarrow \mathbb R$ such that $f\circ \gamma (t) \rightarrow \sup f$ as
$t\rightarrow -\infty$ and  $f\circ \gamma (t) \rightarrow -K > -\infty$ as $t\rightarrow +\infty$. But this is impossible since $f\circ
\gamma$
must be strictly concave. This proves that $f$ is not bounded below and that $f$ is proper when restricted to $\mathbb T^2 \times
[0,+\infty)$.

Now we consider $C_1=\min_{\mathbb T^2 \times \{0\}} f$ and $C_2<C_1<0$ such that the level set $\{f=C_2\}$ has at least one connected
component $S$ in $\mathbb T^2 \times (0,+\infty)$. Then $S$ is closed and bounded since no sequence of points with bounded $f$ can escape
to infinity. Therefore $S$ is compact. More explicitly, all level sets $\{f=C_3\}$ with $C_3<C_2$ contained in $\mathbb T^2 \times
(0,+\infty)$ are compact.

Now we push the level set $S$ to all other level sets by following the flowline $\varphi (x,t)$ of the vector field $\frac{\grad f}{|\grad
f|^2}$. Firstly, the diffeomorphism $\varphi(\cdot,t)$ brings the level set $S \subseteq \{f=C_2\}$ to the level set $\{f=C_2+t\}$,
\begin{align*}
f( \varphi (x,t)) &= f(\varphi(x,0)) + \int_0^t \frac{d}{ds} f(\varphi(x,s)) \ ds = f(x) + \int_0^t \langle \grad f
,\frac{d}{ds}\varphi(x,s) \rangle \ ds  \\
&= f(x) + \int_0^t \langle \grad f , \frac{\grad f }{|\grad f|^2} \rangle \ ds = f(x) + t .
\end{align*}
Secondly, the diameter distorsion between these two level sets is bounded. If $\gamma :[0,1]\rightarrow \{0\}\times \mathbb T^2$ is a
curve on a torus,
\begin{align*}
g_{\varphi_t(x)}(\dot \gamma , \dot \gamma) &= |\dot \gamma|_x^2 + \int_0^t \frac{d}{ds} g_{\varphi_s(x)} (\dot\gamma , \dot\gamma) \ ds \\
&= |\dot \gamma|_x^2 + \int_0^t \mathcal L_{\frac{\grad f}{|\grad f|^2}} (\dot\gamma , \dot\gamma) \ ds \\
&= |\dot \gamma|_x^2 + \int_0^t \frac{2 \Hess f (\dot\gamma , \dot\gamma)}{|\grad f|^2} \ ds .
\end{align*}
Since $\Hess f <0$, this implies that $|\dot \gamma|_{\varphi_t(x)}^2 < |\dot \gamma|_x^2$ so all level sets $\{f=C_4\}$ with $C_4>C_2$
have bounded diameter, and hence are compact and diffeomorphic to $S$.
\end{proof}

Now, the level sets of $f$ are all of them compact and diffeomorphic, thus $M \cong \mathbb R \times \{f=c\} \cong \mathbb
R\times \mathbb T^2$ and therefore the level sets of $f$ are tori. This allows us to set up a coordinate system $(r,x,y)\in \mathbb R
\times \mathbb S^1 \times \mathbb S^1$ such that the potential function $f$ depends only on the $r$-coordinate. Furthermore, the gradient
of $f$ is orthogonal to its level sets, so the metric can be chosen not to contain terms on $dr\otimes dx$ nor $dr\otimes dy$. Thus the
metric can be written $g=u^2 dr^2 + \tilde g$ where $u=u(r,x,y)$ and $\tilde g$ is a family of metrics on the torus with
coordinates $(x,y)$ parameterized by $r$. Using isothermal coordinates, every metric on $\mathbb T^2$ is (globally) conformally
equivalent to the euclidean one, thus $\tilde g = e^{2h}(dx^2 + dy^2)$ where $h=h(r,x,y)$. These conditions allow us to perform computations
that reduce to the particular case we studied in Section \ref{secexist}.

\begin{lema}
Consider the metric over $\mathbb R \times \mathbb T^2$
$$g=u^2 dr^2 + e^{2h} (dx^2 + dy^2)$$
where $u=u(r,x,y)$, $h=h(r,x,y)$, and a function $f=f(r)$. Assume that $g$, $f$ satisfy the soliton equation (\ref{solitoneq}) and that
$g$ has bounded nonconstant curvature. Then $g$ and $f$ are the ones described on the cusped soliton example of Section \ref{secexist}.
\end{lema}

\begin{proof}
The same riemannian computations as before lead us to the soliton equation

\begin{align*}
0 &= \Ric + \Hess f +\frac{\epsilon}{2}g \\
&= \textstyle \frac{1}{u} E_{11} \ dr^2 + \frac{1}{u^3}e^{2h} E_{22} \ dx^2 + \frac{1}{u^3}e^{2h} E_{33} \ dy^2 +
\frac{1}{u} E_{12} \ dr \ dx +\frac{1}{u} E_{13} \ dr \ dy +\frac{1}{u} E_{23} \ dx \ dy
\end{align*}

where

\begin{align*} 
E_{11} =& \textstyle
	 -u^2 \left( \frac{\partial^2 u}{\partial x^2} + \frac{\partial^2 u}{\partial y^2}\right) e^{-2h} +
	 \frac{\epsilon}{2}u^3 + f'' u -2 \left(\frac{\partial h}{\partial r}\right)^2 u -2 \frac{\partial^2 h}{\partial r^2} u +
	 \frac{\partial u}{\partial r}\left( 2\frac{\partial h}{\partial r} -f'\right) ,
\\
E_{22} =& \textstyle
	 -u^3\left(\frac{\partial^2 h}{\partial x^2} + \frac{\partial^2 h}{\partial y^2} \right) e^{-2h} 
	 -u^2\left(\frac{\partial^2 u}{\partial x^2} + \frac{\partial u }{\partial y } \frac{\partial h}{\partial y} - \frac{\partial h}{\partial x} \frac{\partial u}{\partial x} \right) e^{-2h} \\
	  & \textstyle
	 + \frac{\epsilon}{2}u^3 + \frac{\partial h}{\partial r} f' u -2 \left(\frac{\partial h}{\partial r}\right)^2 u - \frac{\partial^2 h}{\partial r^2} u 
	 + \frac{\partial u}{\partial r} \frac{\partial h}{\partial r} ,
\\
E_{33} =& \textstyle
	 -u^3\left(\frac{\partial^2 h}{\partial x^2} + \frac{\partial^2 h}{\partial y^2} \right) e^{-2h} 
	 -u^2\left(\frac{\partial^2 u}{\partial y^2} + \frac{\partial u }{\partial x } \frac{\partial h}{\partial x} - \frac{\partial h}{\partial y} \frac{\partial u}{\partial y} \right) e^{-2h} \\
	  & \textstyle
	 + \frac{\epsilon}{2}u^3 + \frac{\partial h}{\partial r} f' u -2 \left(\frac{\partial h}{\partial r}\right)^2 u - \frac{\partial^2 h}{\partial r^2} u 
	 + \frac{\partial u}{\partial r} \frac{\partial h}{\partial r} ,
\\
E_{12} =& \textstyle
	 \frac{\partial u}{\partial x} \left( \frac{\partial h}{\partial r} -f' \right) - u \frac{\partial^2 h}{\partial x \partial r} ,
\\
E_{13} =& \textstyle
	 \frac{\partial u}{\partial y} \left( \frac{\partial h}{\partial r} -f' \right) - u \frac{\partial^2 h}{\partial y \partial r} ,
\\
E_{23} =& \textstyle
	 \frac{\partial u}{\partial y}\frac{\partial h}{\partial x} + \frac{\partial u}{\partial x}\frac{\partial h}{\partial y} 
	 - \frac{\partial^2 u}{\partial x \partial y} .
\end{align*} 
Since the function $u$ never vanishes, nor the exponential does, the soliton equation is the PDE system
$\{E_{11}=E_{22}=E_{33}=E_{12}=E_{13}=E_{23}=0\}$. It is convenient to substitute the equations $E_{22}=0$ and $E_{33}=0$ with the linearly
equivalent $-\frac{1}{2}(E_{22}+E_{33})=0$ (equation (\ref{eq23a}) below) and $E_{22}-E_{33}=0$ (equation (\ref{eq23b})). Then, we get the
system
\begin{align} 
&	 -u^2 e^{-2h} \bigtriangleup u +
	 \frac{\epsilon}{2}u^3 + f'' u -2 \left(\frac{\partial h}{\partial r}\right)^2 u -2 \frac{\partial^2 h}{\partial r^2} u +
	 \frac{\partial u}{\partial r}\left( 2\frac{\partial h}{\partial r} -f'\right) =0 \label{eq1} ,
\\
&	 -u^3 e^{-2h} \bigtriangleup h -\frac{1}{2} u^2 e^{-2h} \bigtriangleup u 
	 + \frac{\epsilon}{2}u^3 + \frac{\partial h}{\partial r} f' u -2 \left(\frac{\partial h}{\partial r}\right)^2 u - \frac{\partial^2 h}{\partial r^2} u 
	 + \frac{\partial u}{\partial r} \frac{\partial h}{\partial r} =0 \label{eq23a} ,
\\
&	\frac{\partial u}{\partial x} \left( \frac{\partial h}{\partial r} -f' \right) - u \frac{\partial^2 h}{\partial x 		
\partial r} =0 \label{eq12} ,
\\
&	\frac{\partial u}{\partial y} \left( \frac{\partial h}{\partial r} -f' \right) - u \frac{\partial^2 h}{\partial y \partial 	r}
=0 \label{eq13} ,
\\
&	\frac{\partial u}{\partial y}\frac{\partial h}{\partial x} + \frac{\partial u}{\partial x}\frac{\partial h}{\partial y} 
	 - \frac{\partial^2 u}{\partial x \partial y} =0 \label{eq23} ,
\\
&	 2 \frac{\partial u}{\partial x}\frac{\partial h}{\partial x} -2 \frac{\partial u}{\partial y}\frac{\partial h}{\partial 	
y}  + \frac{\partial^2 u}{\partial x^2} - \frac{\partial^2 u}{\partial y^2} =0 \label{eq23b} ,
\end{align}
where $\bigtriangleup = \frac{\partial^2}{\partial x^2} + \frac{\partial^2}{\partial y^2} $ is the euclidean laplacian on the $xy$-surface.
We will recover our cusped soliton proving that $u\equiv 1$ and that $h(r,x,y)$ actually does not depend on $(x,y)$.

We consider first the equations (\ref{eq23}) and (\ref{eq23b}).
Since no derivatives on $r$ are present, we can consider the problem for $r$ fixed, so $u =u(r,\cdot,\cdot)$ is a function on the $xy$-torus with metric $e^{2h(r,\cdot,\cdot)}(dx^2+dy^2)$. The function $u$ must have extrema over the torus, since it is compact, so there are some critical points $(x_i,y_i)$ such that $\frac{\partial u}{\partial x}\big|_{(x_i,y_i)} = \frac{\partial u}{\partial x}\big|_{(x_i,y_i)}=0$. From the equations evaluated on a critical point, $\frac{\partial^2 u}{\partial x \partial y}\big|_{(x_i,y_i)}=0$ and $\frac{\partial^2 u}{\partial x^2}\big|_{(x_i,y_i)} =\frac{\partial^2 u}{\partial y^2}\big|_{(x_i,y_i)} = \lambda_i$ so the Hessian matrix (on the $xy$-plane) is 
$$\left(\begin{array}{cc} \lambda_i & 0 \\ 0 & \lambda_i  \end{array}\right) .$$
Suppose that every critical point is nondegenerate, that is, the Hessian matrix is nonsingular with $\lambda_i\neq 0$. Then the set of
critical points is discrete and $u$ is a Morse function for the torus. But then the Morse index on every critical point (the number of
negative eigenvalues of the Hessian) is either $0$ or $2$, meaning that every critical point is either a minimum or a maximum, never a
saddle point. Then Morse theory implies that the topology of the $xy$-surface cannot be a torus (being actually a sphere, see \cite{Mil}).
This
contradicts that every critical point is nondegenerate, so there is some point $(x_0,y_0)$ such that first and second derivatives vanish. 

We now proceed to derivate the two equations. Equations (\ref{eq23}) and (\ref{eq23b}) can be written
\begin{align}
u_{xy} &= u_y h_x + u_x h_y    \label{eq23_p} ,\\
u_{xx} - u_{yy} &= -2 u_x h_x +2 u_y h_y   \label{eq23b_p} ,
\end{align}
using subscripts for denoting partial derivation. Their derivatives are
\begin{align*}
u_{xxy} &= u_{xy}h_x + u_y h_{xx} + u_{xx}h_y + u_x h_{xy} ,\\
u_{xyy} &= u_{yy}h_x + u_y h_{xy} + u_{xy}h_y + u_x h_{yy} ,\\
u_{xxx} - u_{xyy} &= -2 u_{xx}h_x -2 u_x h_{xx} +2 u_{xy} h_y +2 u_y h_{xy} ,\\
u_{xxy} - u_{yyy} &= -2 u_{xy}h_x -2 u_x h_{xy} +2 u_{yy} h_y +2 u_y h_{yy} ,
\end{align*}
using the same notation. Evaluated at the point $(x_0,y_0)$, where all first and second order derivatives of $u$ vanish, the right-hand side
of these equations vanish and therefore all third derivatives vanish. Inductivelly, if all $n$-th order derivatives vanish at $(x_0,y_0)$,
then the $(n-1)$-th derivative of the equation (\ref{eq23_p}) implies that all mixed $(n+1)$-th order derivatives (derivating at least once
in each variable) vanish, then the $(n-1)$-th derivative of the equation (\ref{eq23b_p}) implies that all pure $(n+1)$-th order derivatives
(derivating only in one variable) also do; so all derivatives of all orders of $u$ vanish at $(x_0,y_0)$. Because $u(r,x,y)$ is a component
of a solution of the Ricci flow, it is an analytical function (see \cite{TRFTAA2}, Ch. 13), so it must be identically constant in $(x,y)$.

At this point, we can reduce our metric to be $g=u(r)^2 dr^2 + e^{2h}(dx^2+dy^2)$ with $h=h(r,x,y)$. It is just a matter of reparameterizing the variable $r$ to get a new variable, $\bar r = \int u(r) \ dr$, such that $u(r)^2dr^2 = d\bar r^2$, so we rename $\bar r$ as $r$ and we can assume that the metric is $g=dr^2 + e^{2h}(dx^2+dy^2)$ with $h=h(r,x,y)$. 

We now look at the equations (\ref{eq12}) and (\ref{eq13}) when $u\equiv 1$, they imply
$$\frac{\partial^2 h}{\partial x \partial r} = \frac{\partial^2 h}{\partial y \partial r} = 0 ,$$
meaning that $\frac{\partial h}{\partial r}$ does not depend on $x$, $y$. 
Finally, looking at equations (\ref{eq1}), (\ref{eq23a}) when $u\equiv 1$, we get
\begin{align*}
\frac{\epsilon}{2} + f'' -2 \left(\frac{\partial h}{\partial r}\right)^2 -2 \frac{\partial^2 h}{\partial r^2} &=0 
,\\ 
\frac{\epsilon}{2} + \frac{\partial h}{\partial r} f' -2 \left(\frac{\partial h}{\partial r}\right)^2 - \frac{\partial^2 h}{\partial r^2} &=
e^{-2h} \bigtriangleup h .
\end{align*}
Since the left-hand side does not depend on $(x,y)$, nor does the term $e^{-2h} \bigtriangleup h$. Recall that a two-dimensional metric
written as $e^{2h(x,y)}(dx^2+dy^2)$ has gaussian curvature $K=-e^{-2h} \bigtriangleup h$. So the $xy$-tori have each one constant curvature,
and the only admitted one for a torus is $K=0$. Hence $h$ only depends on $r$ and the equations
turn into the system (\ref{eqssolcusp}), that we already studied for the example of the cusp soliton. The rest of the uniqueness follows
from the discussion on Section \ref{secexist}.
\end{proof}

\section{Evolution of curvature} \label{seccurv}

On this last section we expose the property anounced in Theorem \ref{tma_evolcur}, derivated from the opposite effects of the
diffeomorphism and the homothety for the evolution of the metric. Recall that $(M,g(t))$ is the (soliton) Ricci flow defined on $M=\mathbb R
\times \mathbb T^2 $ and for $t\in(-1,+\infty)$, such that $g(0)=g_0$  where $g_0$ is the metric constructed in Theorem \ref{tma_exist}, and
let us denote $R=R(t)$ the scalar curvature of $g(t)$. We want to show that the growth of the curvature along $(M,g(t))$ changes
sign for values of $t$ far enough of $-1$, but is positive everywhere along the manifold for values of $t$ close enough to $-1$.

\begin{proof}[Proof (of Theorem \ref{tma_evolcur}).]
The evolution of the soliton metric under the Ricci flow is
$$g(t)=(t+1)\phi_t^*(g_0)$$
where $g_0$ is the metric constructed in Theorem \ref{tma_exist} and $\phi_t(r_0,x_0,y_0)=(r(t),x_0,y_0)$ with 
$$\left\{\begin{array}{rcl}\dot r(t) &=& f'(r(t)) = F(r(t)) \\ r(0) &=& r_0 . \end{array} \right. $$
Then, since $R[g_0]=-4H'-6H^2$ (by Lemma \ref{lema_qtties}),
$$R[g(t)]_{(r_0,x_0,y_0)}= \frac{1}{t+1} R[g_0]\big|_{(r(t),x_0,y_0)} = \frac{1}{t+1}(-4H'-6H^2) \big|_{r=r(t)} ,$$
so
\begin{align*}
 \frac{d}{dt}R[g(t)]_{(r_0,x_0,y_0)} &= \frac{-1}{(t+1)^2}(-4H'-6H^2) + \frac{1}{t+1}(-4H''F-12HH'F) \big|_{r=r(t)} \\
&=\frac{2}{(t+1)^2}\left[ (2HF-H^2+1) + (t+1) F^2 (-2HF+2H^2-1) \right]\big|_{r=r(t)} .
\end{align*}

Thus, the zeroset $\{\frac{d}{dt}R[g(t)]=0\}$ defines, for each $t$, an algebraic curve on the $HF$-plane. If the solution curve $S$
intersects this zeroset curve, then
the soliton changes the growth sign of the curvature at some point. Otherwise $R$ is everywhere monotone.

For the rest of the proof, we rename the variables so our system (\ref{eqsphase}) is
\begin{equation}
\left\{ \begin{array}{rcl}
          \dot x & = & xy -2 x^2 + \frac{1}{2} \\
          \dot y & = & 2xy -2 x^2 + \frac{1}{2} 	 
         \end{array} \label{eqsphasexy}
\right. 
\end{equation}
where $x=x(r)$, $y=y(r)$, the curve $S$ is the separatrix solution of the system (\ref{eqsphasexy}) emanating from the critical point
$(\frac{1}{2},0)$ towards the vertical asymptote $x=0$, and
\begin{equation} 
 C_t = (2xy-x^2+1) + (t+1) y^2 (-2xy+2x^2-1) . \label{Cxyori}
\end{equation}
The question is whether $\{C_t=0\}$ intersects $S$, for each $t\in (-1,+\infty)$.
Figure \ref{imgblows_a} represents the solution curve $S$ together with the curve $C_t$ for $t=10$. This gives some
evidence that for big values of $t$ there is an intersection point of the curves, but it is not clear for small or negative values of $t$.
The issue is that $C_t$ has an asymptote and whether it approaches the infinity at the right hand side or at the left hand side of
the curve $S$.
In order to study these guesses, we perform again a projective change of variables, equivalent to assume that our phase
portrait lies on the $\{z=1\}$ plane (with coordinates $(x,y,1)$) of the $xyz$-space, and we project perspectively from the origin to the $\{y=-1\}$ plane
(with coordinates $(\tilde x, -1, \tilde y)$). 
$$
\left\{ \begin{array}{rcl} 
         \tilde x &=& -\frac{x}{y} \\
	 \tilde y &=& -\frac{1}{y}
        \end{array}
\right.
\quad , \quad
\left\{ \begin{array}{rcl}
         x &=& \frac{\tilde x}{\tilde y} \\
	 y &=& -\frac{1}{\tilde y}
        \end{array}
\right. .
$$

This change of coordinates has the effect of
bringing the point at the infinity on the vertical asymptote to the new origin of coordinates, the old line at infinity to the horizontal
axis, and
the old (projective) line $y=0$ to the new line of the infinity. We won't keep track of the tilde notation and use again $x$, $y$ as
coordinates.

After this change, the system turns into

\begin{equation}
\left\{ \begin{array}{rcl}
          \dot x & = & \frac{-8x^2-4x^3 + xy^2 -2x+y^2}{2y} \\
          \dot y & = & -2x-2x^2+\frac{1}{2}y^2	 
         \end{array} \label{eqsphasexyinf}
\right. 
\end{equation}
that is equivalent (has the same orbits) to the system
\begin{equation}
\left\{ \begin{array}{rcl}
          \dot x & = & -4x^2-2x^3 + \frac{1}{2}xy^2 -x+\frac{1}{2}y^2 \\
          \dot y & = & -2xy-2x^2y+\frac{1}{2}y^3	 
         \end{array} \label{eqsphasexyinf2}
\right. ,
\end{equation}
and the curve $C_t$ turns into
$$ \frac{-2xy^2-x^2y^2+y^4+(t+1)(2x+2x^2-y^2)}{y^4}$$
that has the same zeroset as
\begin{equation}
  C_t = -2xy^2-x^2y^2+y^4+(t+1)(2x+2x^2-y^2) . \label{Cxyinf}
\end{equation}
See Figure \ref{imgblows_b}. In particular, we can check that $C_t$ now passes through the origin, and this confirms that the original $C_t$
in (\ref{Cxyori}) had a vertical asymptote. Nevertheless, it is not yet clear which curve lies at which side near the contact point. In
order to
investigate this behaviour, we
perform some algebraic blow-ups at the contact point. Recall that an algebraic blow-up is a change of variables from the old $(x,y)$
to the new ($\tilde x, \tilde y)$  given by
$$
\left\{ \begin{array}{rcl} 
         \tilde x &=& \frac{x}{y} \\
	 \tilde y &=& y
        \end{array}
\right.
\quad , \quad
\left\{ \begin{array}{rcl}
         x &=& \tilde x \tilde y \\
	 y &=& \tilde y
        \end{array}
\right. .
$$
The mapping $\varphi : (\tilde x, \tilde y) \mapsto (x,y)$ is a birrational map, which restricts to a diffeomorphism in all points except at
$(x,y)=(0,0)$, where $\varphi^{-1}((0,0)) = \{\tilde y=0\}$ is a (projective) line called the exceptional divisor of the blow-up. The
exceptional divisor is in correspondence with the space of directions of the old origin, thus we ``pick out a point and substitute it with a
projective line''. Again, we won't keep track of the tildes. Two curves intersecting with
normal crossing at the origin are transformed in this way to two separated curves; two curves tangent at the origin, when transformed,
still intersect, but their contact order is decreased. Since all the curves we are involved with are analytical, after a finite number of
blow-ups the process finishes separating the curves. With the exception of the multiply blown-up line $\{y=0\}$, all the remaining phase
portrait is diffeomorphic to the original one, so any intersecting point other than the origin will still be present in the blown-up
portrait.

The process can be algorithmically carried on. We consider the vector field of the system (\ref{eqsphasexyinf2}). The solution $S$
intersects the $\{y=0\}$ axis at the (one) critical point of the vector field, that can be symbolically computed. We consider also the curve
$C_t$ in (\ref{Cxyinf}) after the chart change. This is a polynomial in $x,y$ and its intersection with $\{y=0\}$ can also be computed
symbolically. We
perform the change of variables corresponding to the blow-up, and ocasionally translate the new intersection point to the origin again. We
compute both intersection points with $\{y=0\}$ and iterate up to when the two results disagree. Let us remark that this process can be
carried out by a symbolic algorithm, so there is no numerical approximation involved. See Figure \ref{imgblows} for a numerical
visualization.

Once this is done, we find that after six blow-ups the critical point of the system is located at $(0,0)$, and the intersection of $C_t$
with the $\{y=0\}$ line is at $(\frac{1}{8} \frac{t}{t+1},0)$. Thus, generically the curve $C_t$ and the solution curve $S$ have order of
contact five at the infinity in the original phase portrait of (\ref{eqsphasexy}). In the case $t=0$, both points of intersection agree, so
further blow-ups are
needed. It turns out that the tenth blow-up separates the points, and when locating the critical point at $(0,0)$, the intersection of $C_t$
with $\{y=0\}$ is at $(\frac{1}{8},0)$. Thus the curve $C_0$ has order of contact nine with the solution $S$ at the infinity in the original
phase portrait of (\ref{eqsphasexy}).

Now we recall a couple of properties of the blow-ups: firstly, curves crossing the origin with a slope $\lambda$ are blown-up to curves
crossing the  $\{y=0\}$ line at $x=\frac{1}{\lambda}$, in particular positive slopes are sent to the $x>0$ half-line, and negative slopes to
the $x<0$ half-line. Secondly, the blow-up preserves orientation of horizontal lines in the upper half-plane, and reverses it in the lower
one. 

The change of charts we performed before the blow-ups preserves the orientation of all horizontal lines, but exchanges the lower and the upper half-planes. In
summary, the relative position (left and right) of $C_t$ and $S$ on the lower half-plane of the original phase portrait of
(\ref{eqsphasexy}), is the same as
on the upper half-plane on the phase portrait after all the blow-ups. 

Therefore we can deduce that for $t \in (-1,0)$ the curve $C_t$ approaches the infinity at the asymptote from the left-hand side of the
separatrix $S$, and for $t \in [0,+\infty)$ it approaches from the right. Given that this component of the $C_t$ curve has always points at
the left-hand side of $S$, we can deduce that $C_t$ intersects $S$ at least at one point (other than the infinity) for $t\geq 0$.

Furthermore, this finite intersection point depends continuously on $t$, so a small perturbation on $t$ will still make the two curves
intersect. Thus, there exist a small $\delta<0$ such that for $t>\delta$ the function $C_t$ changes sign along the separatrix $S$. Actually,
for $\delta < t < 0$ the sign must change at least twice.
\vspace{1em}

Finally, we show that this change of sign of $C_t$ does not happen for some $t$ close enough to $-1$. This is due to the fact that for such $t$
the curve $\{C_t=0\}$ is a barrier for the separatrix. We compute the normal vector to the curve $C_t$ in (\ref{Cxyori}),
$$\grad C_t = \left( 2y - 2x+ (t+1) y^2 (-2y+4x) , 2x + 2(t+1) y (-3xy+2x^2-1)\right)$$
and compare with the vector field of the system (\ref{eqsphasexy})
$$V=\left( xy -2 x^2 + \frac{1}{2} , 2xy -2 x^2 + \frac{1}{2} \right) .$$
Their scalar product is
 \begin{align*}
  & \langle \grad C_t , V \rangle  = \\  & = -y\left(-2xy+2x^2-1 + (t+1)(2xy^3+4x^2y^2+y^2-12x^3y+5xy+8x^4-6x^2+1)\right)
\end{align*}
Restricted to the curve $\{C_t=0\}$, this simplifies by substracting the equation $yC_t=0$,
\begin{equation}
   \psi_t(x,y):=\langle \grad C_t , V \rangle\big|_{C_t}   = -y\left( x^2 + (t+1) (6x^2y^2-12x^3y+5xy+8x^4-6x^2+1) \right) \label{prxy}
\end{equation}
Fortunately $C_t=0$ is a second order equation for $x$, thus it is easy to select the apropriate isolation $x=x(y)$ corresponding to the
branch on $x>0$, $y<0$, 
$$x=x(y)={\frac {-y+ \left( t+1 \right) {y}^{3}+\sqrt {{y}^{2}-2\, \left( t+1
 \right) {y}^{4}+ \left( t+1 \right) ^{2}{y}^{6}-3\, \left( t+1
 \right) {y}^{2}+2\, \left( t+1 \right) ^{2}{y}^{4}+1}}{2\, \left( t+1
 \right) {y}^{2}-1}} ,$$
and substitute it on (\ref{prxy}) (although unfortunately, the explicit expression is quite ugly),
$$ \Psi_t(y):=\psi_t(x(y),y) = \ldots $$
This is just a real function $\Psi_t:\left(-\infty, \frac{-1}{\sqrt{(t+1)}}\right] \longrightarrow \mathbb R$ (the upper bound on
the domain is the negative solution of $C_t(0,y)=0$). This function gives for each value of $y$ the scalar product between the normal
vector to the $C_t$ curve (pointing rightwards) and the vector field of the system at the point $\left(x(y),y\right)$. If this function is strictly positive for values of $t$ close to $-1$, this implies that $C_t$ is a barrier for the separatrix. For, the
separatrix $S$ emanates from $(1/2,0)$, which is on the right hand side of $C_t$, and if $S$ touched $C_t$, then its tangent vector (the vector
field of the system) would be pointing to the same region separated by $C_t$. Otherwise, if the function $\Psi_t$ fails to be
positive, then $C_t$ fails to be a barrier. As shown in Figure \ref{imgisbarrier}, $C_t$ is a barrier for $t=-0.7$ but it is not for $t=-0.2$. 
\begin{figure}[h]
\subfigure[$t=-0.7$]{
  \includegraphics[width=0.3\textwidth]{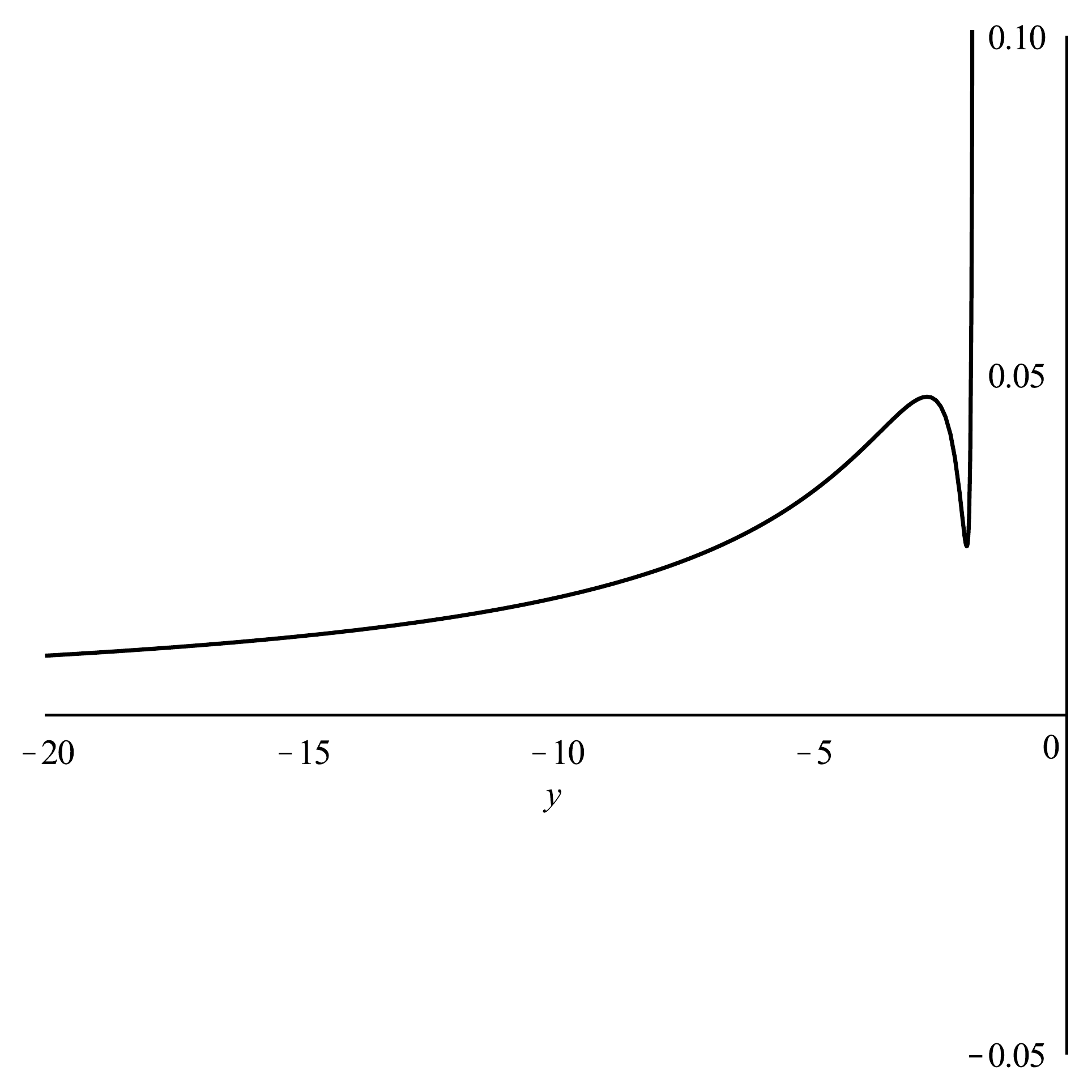} }
\subfigure[$t=-0.2$]{
  \includegraphics[width=0.3\textwidth]{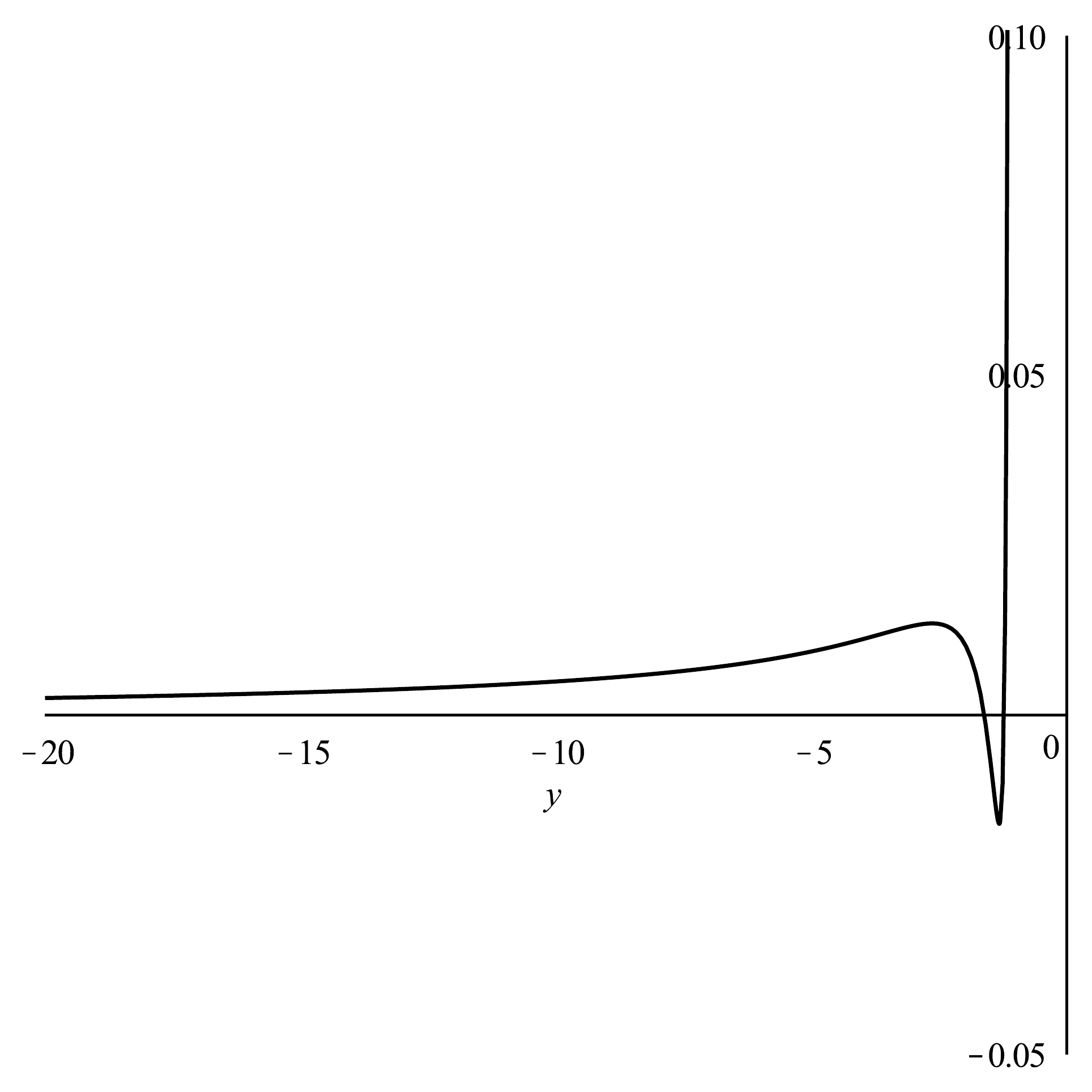} }
\caption{Graph of $\Psi_t(y)$ for two values of $t$.}
\label{imgisbarrier}
\end{figure}
The given value of $-0.7$ is just an example, it could be checked for smaller values. However, it is not immediate to tell which is the critical value, since the failure of $C_t$ to be a barrier does not ensure an actual crossing of $C_t$ and the separatrix $S$.

We can therefore say that the scalar curvature is always negative, and that (for instance) for $t=-0.7$ the soliton is evolving with the
scalar curvature everywhere increasing, whereas for $t \geq 0$ the soliton has regions where the scalar curvature is increasing and regions
where it is decreasing.  Any fixed point eventually belongs to the region of increasing curvature and therefore any point eventually tends to
zero curvature, due to the dominance of the expanding effect.
\end{proof}

\begin{figure}[p]
\newlength{\WID}
\setlength{\WID}{0.25\textwidth}

\subfigure[The separatrix $S$ and the curve $C_{10}$.]{ 
  \includegraphics[width=0.32\textwidth]{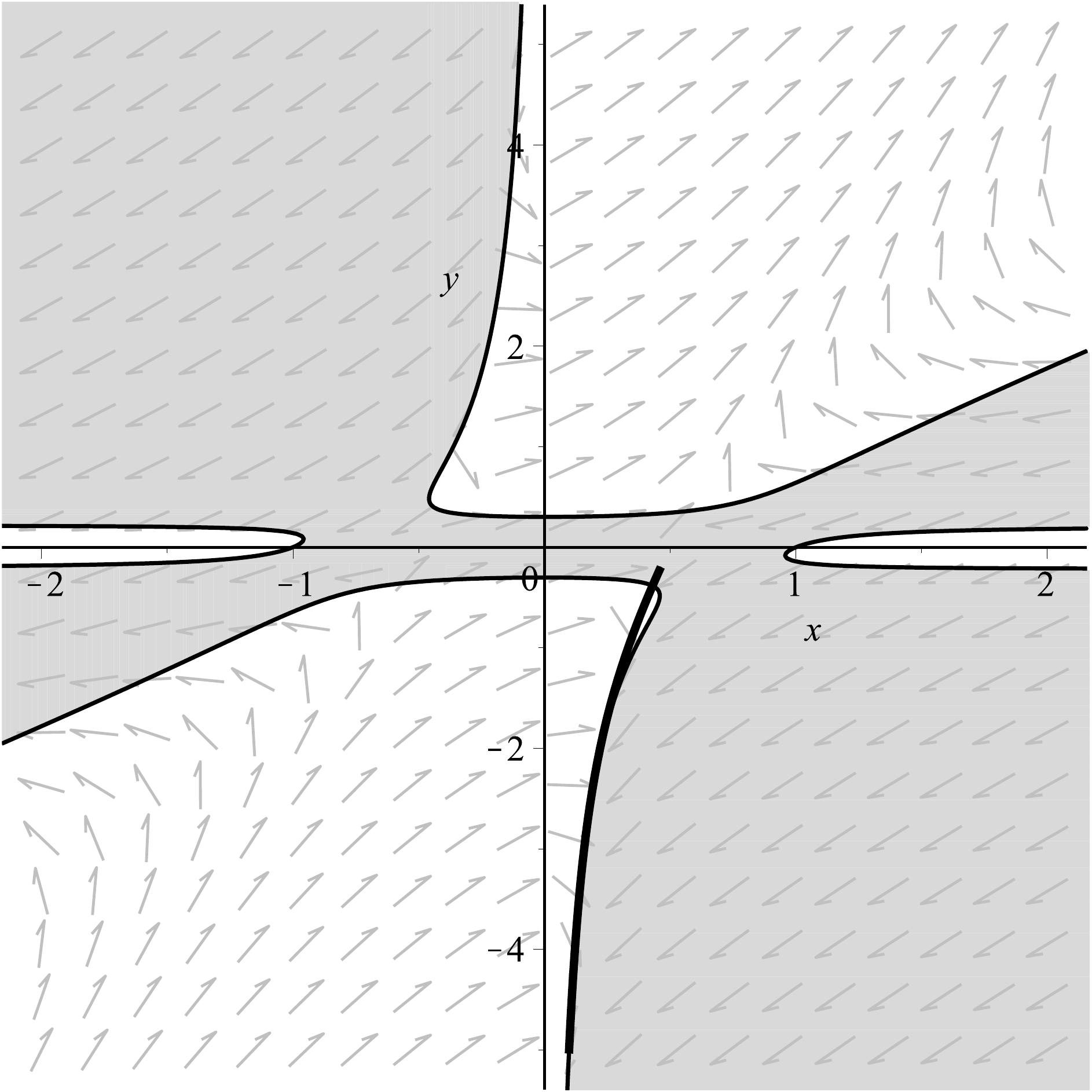}  \label{imgblows_a} }
\subfigure[The point of contact at infinity brought to the origin.]{
  \includegraphics[width=0.32\textwidth]{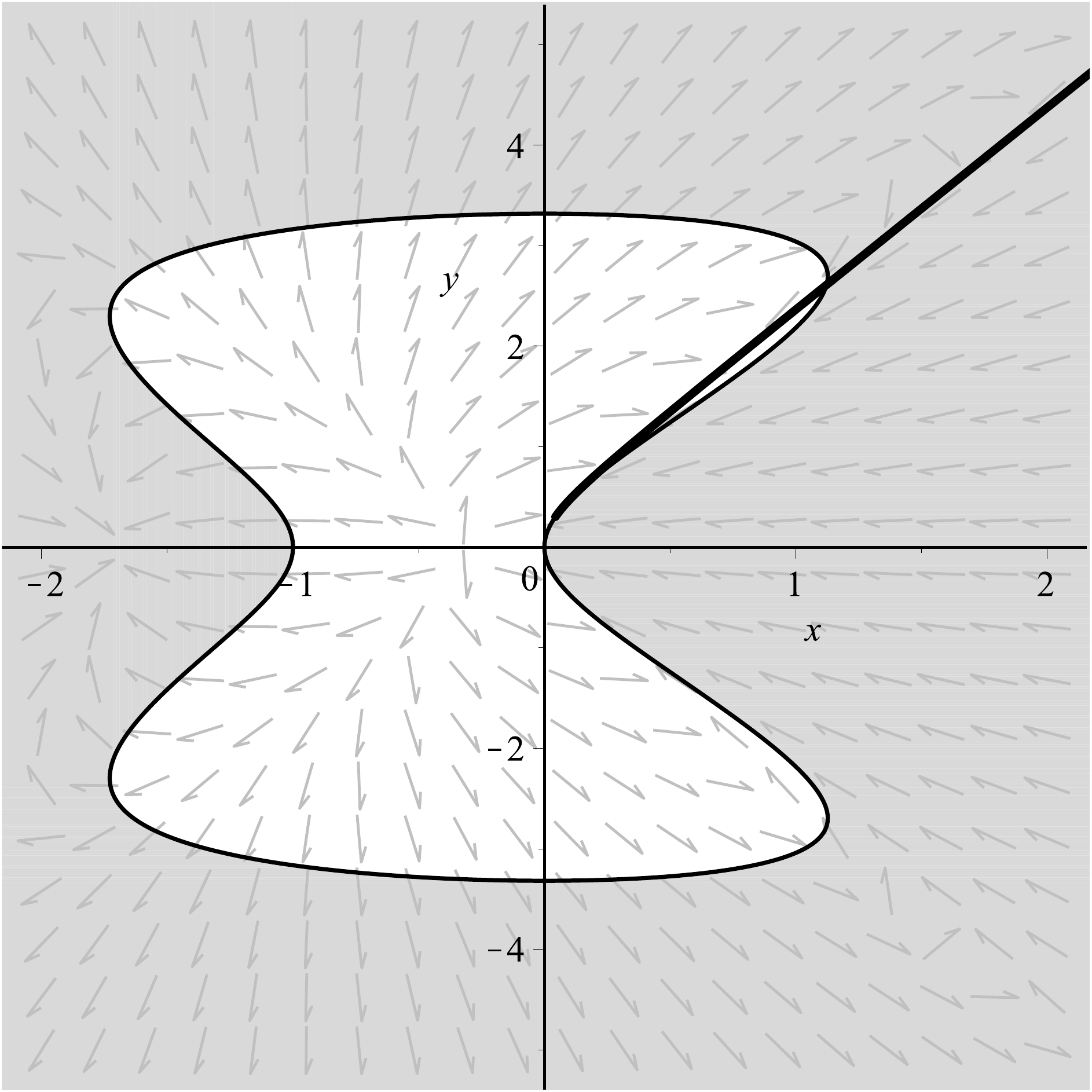} \label{imgblows_b} }

\subfigure[First blow-up.]{
\includegraphics[width=\WID]{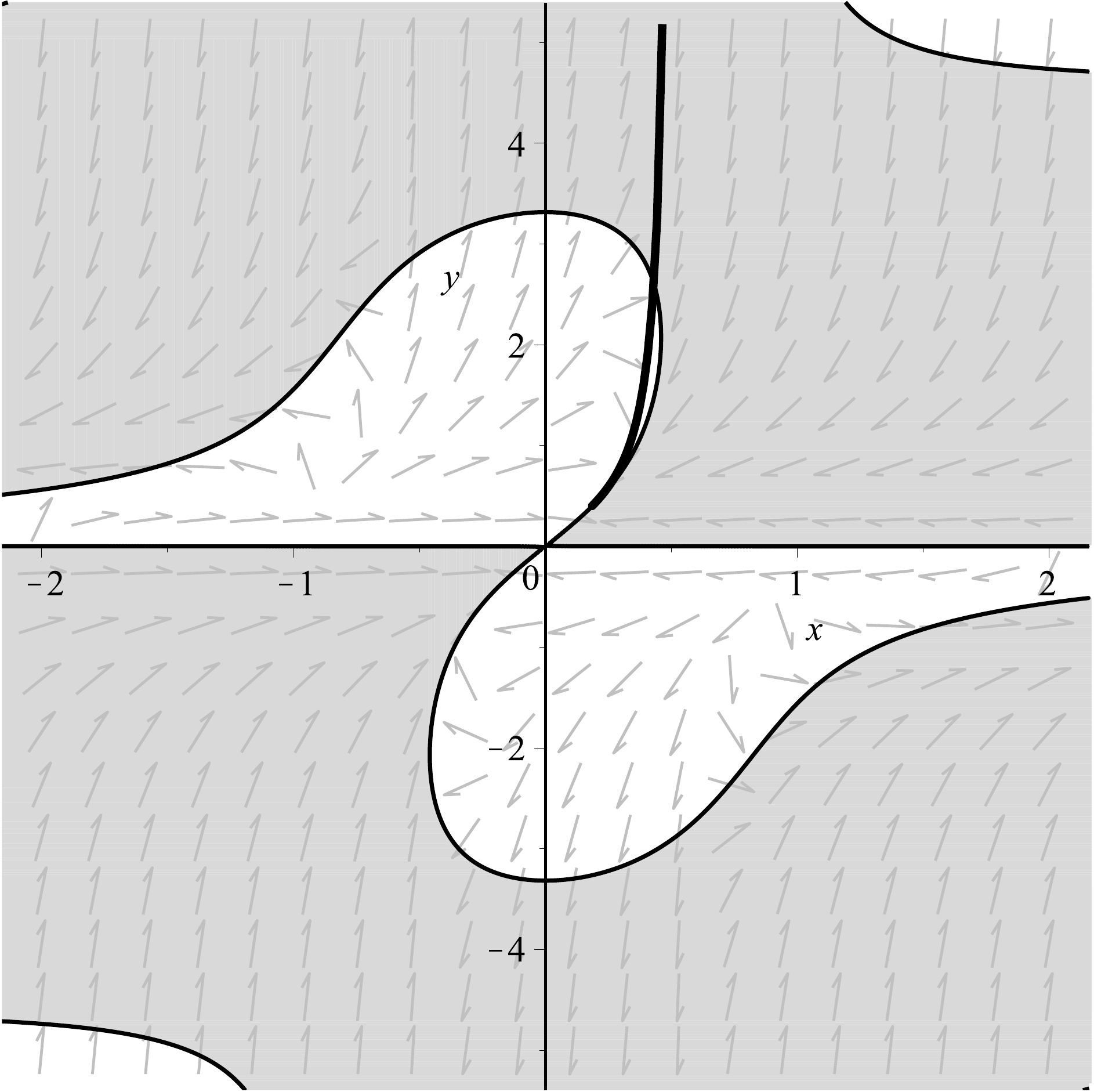} }
\subfigure[Second blow-up.]{
\includegraphics[width=\WID]{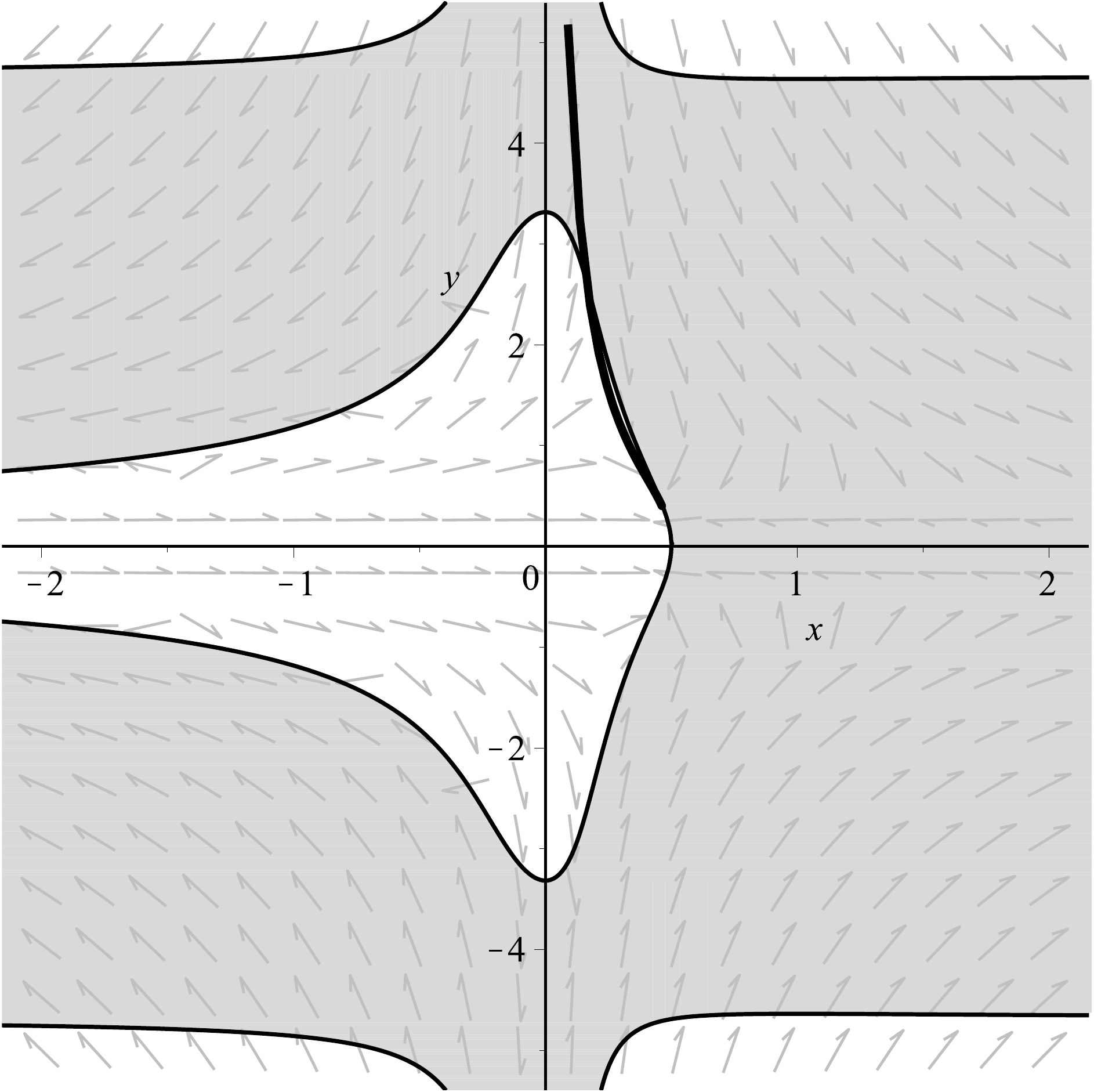} }
\subfigure[Translation.]{
\includegraphics[width=\WID]{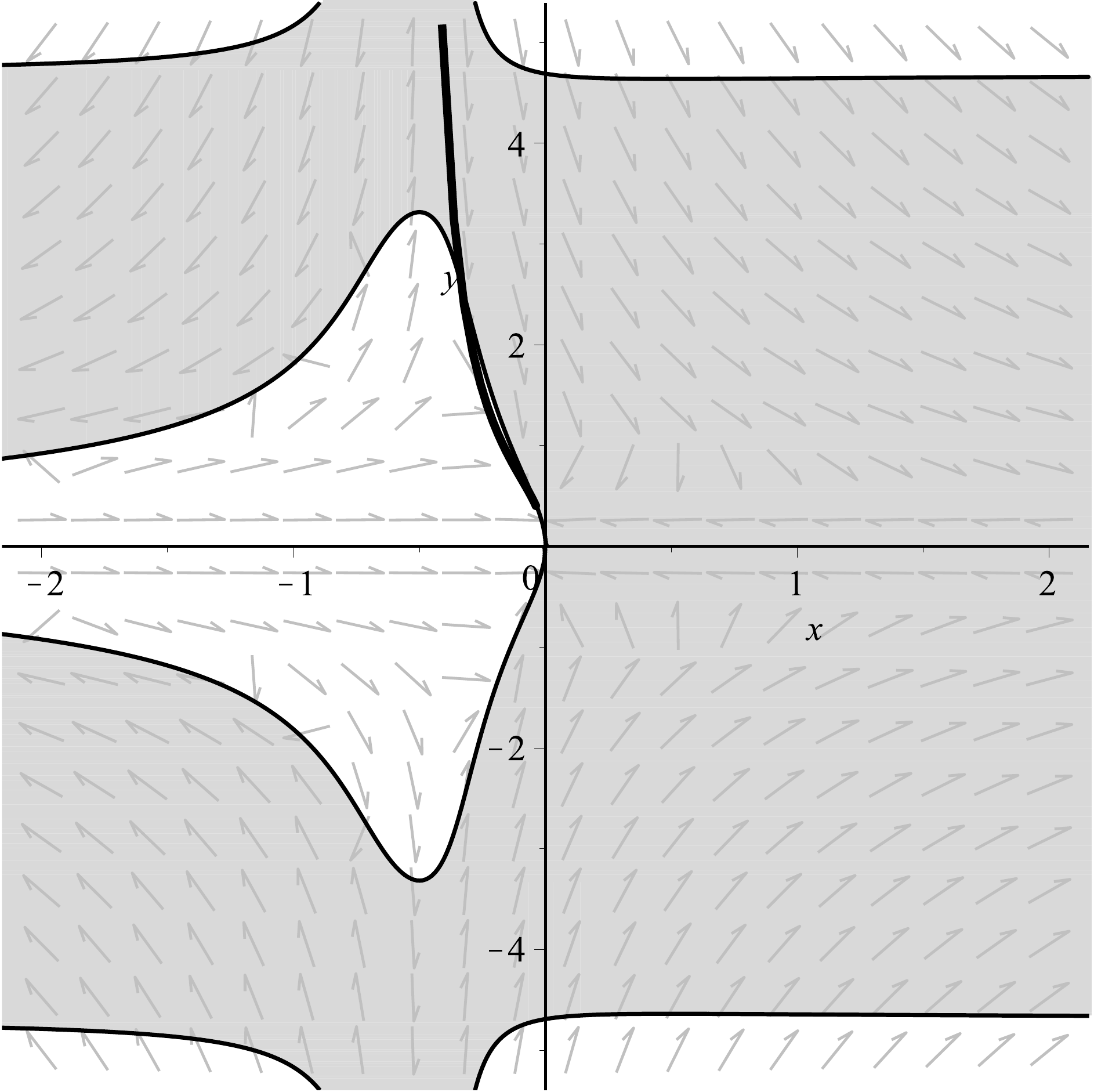} }
\subfigure[Third blow-up.]{
\includegraphics[width=\WID]{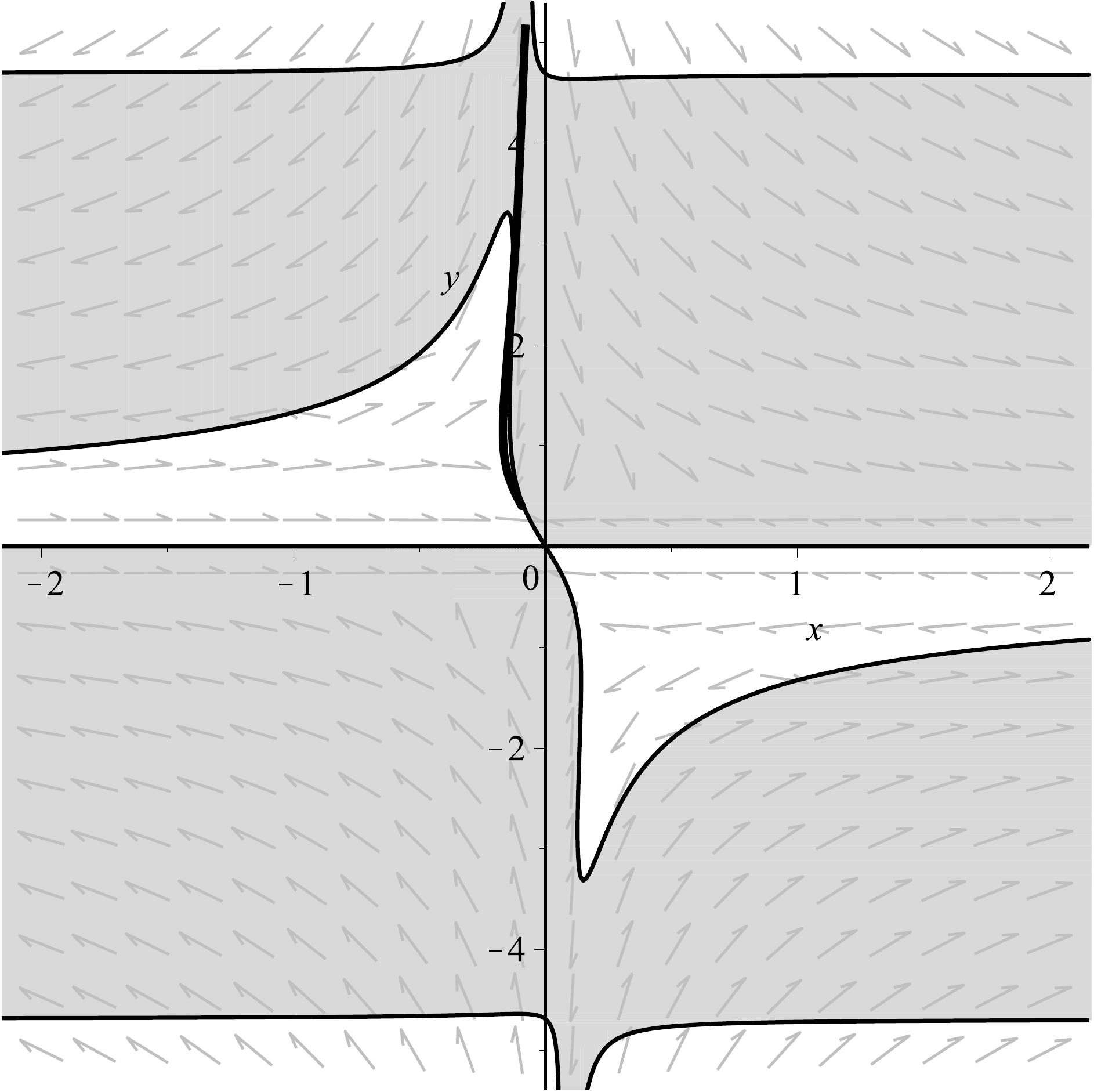} }
\subfigure[Fourth blow-up.]{
\includegraphics[width=\WID]{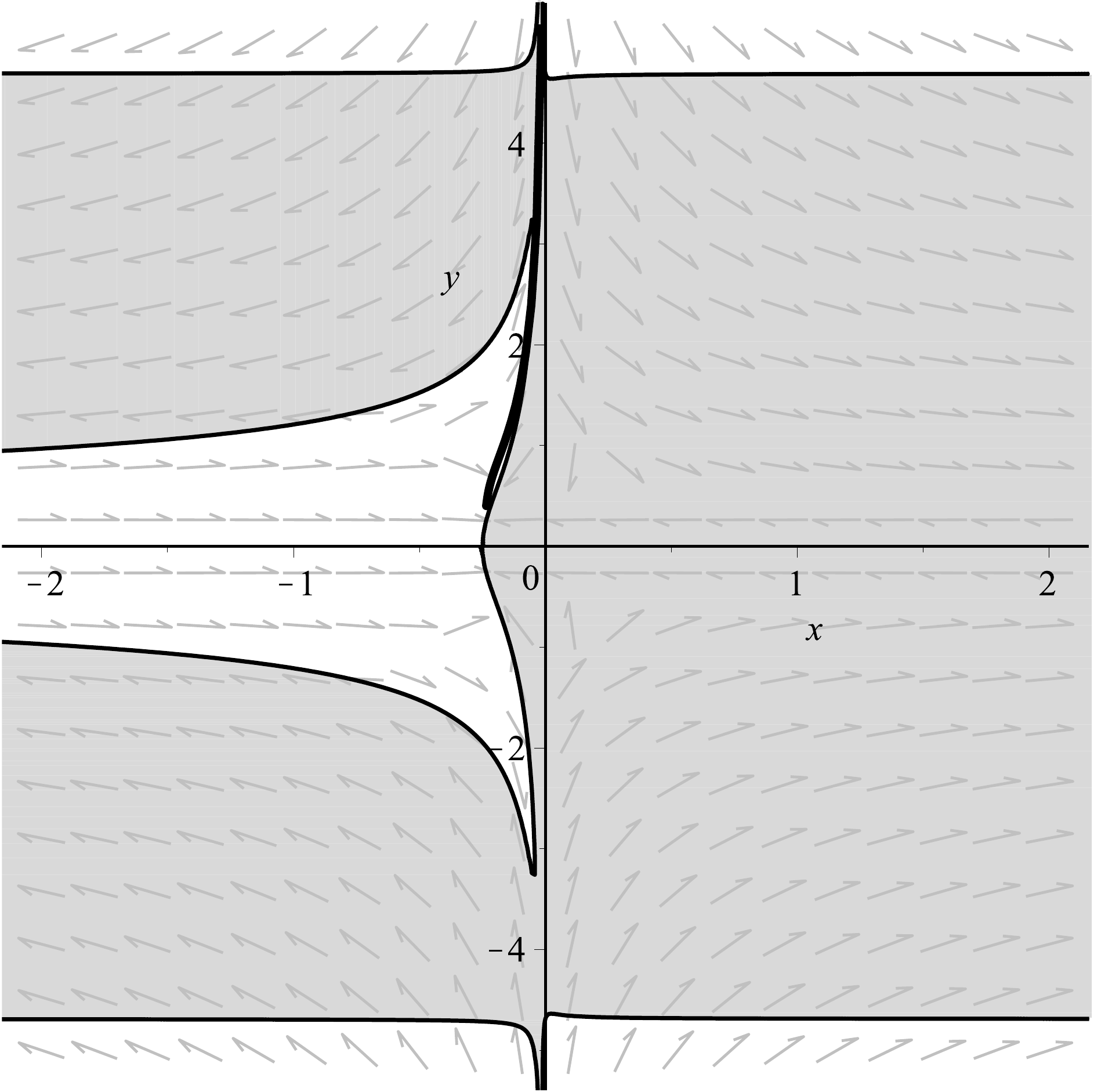} }
\subfigure[Translation.]{
\includegraphics[width=\WID]{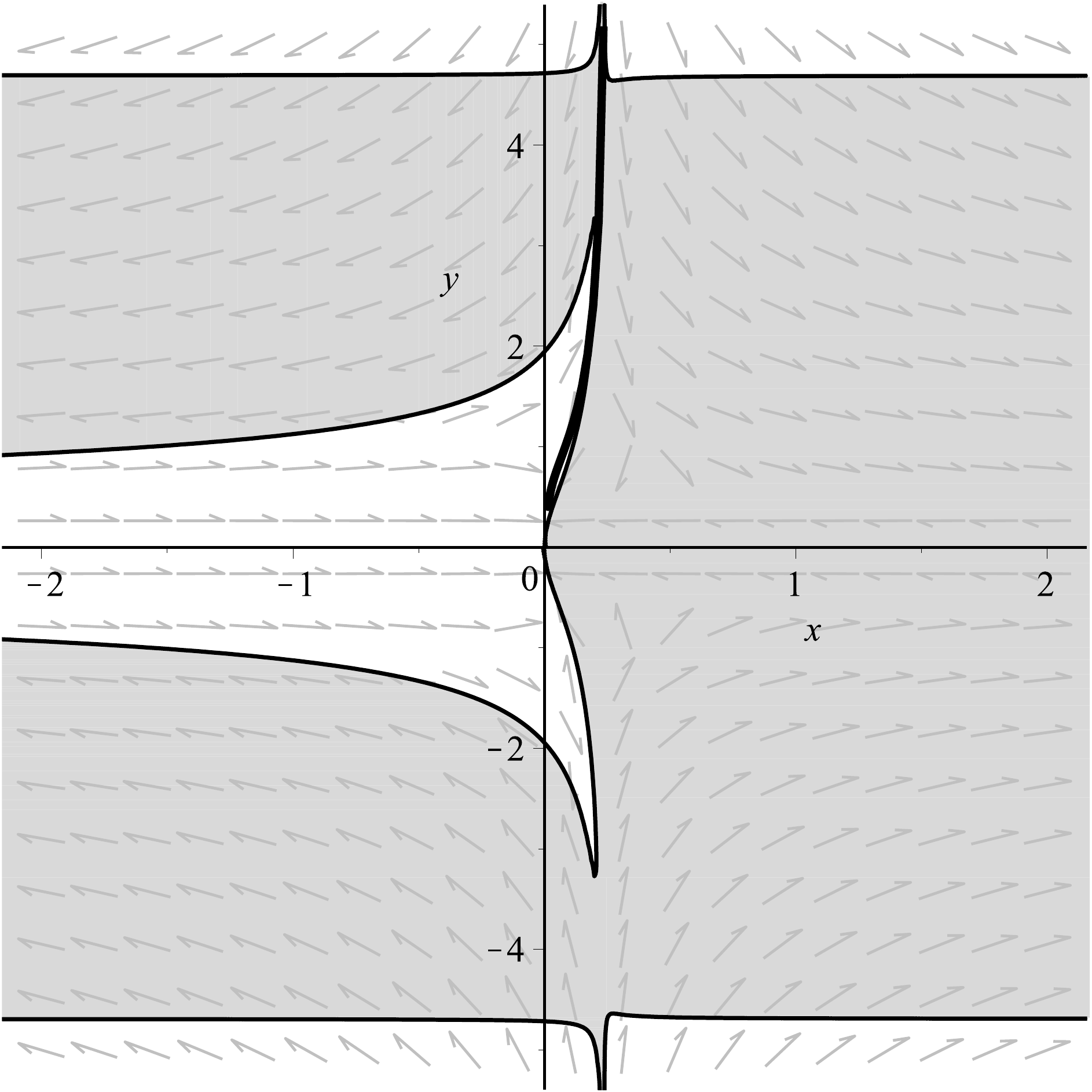} }
\subfigure[Fifth blow-up.]{
\includegraphics[width=\WID]{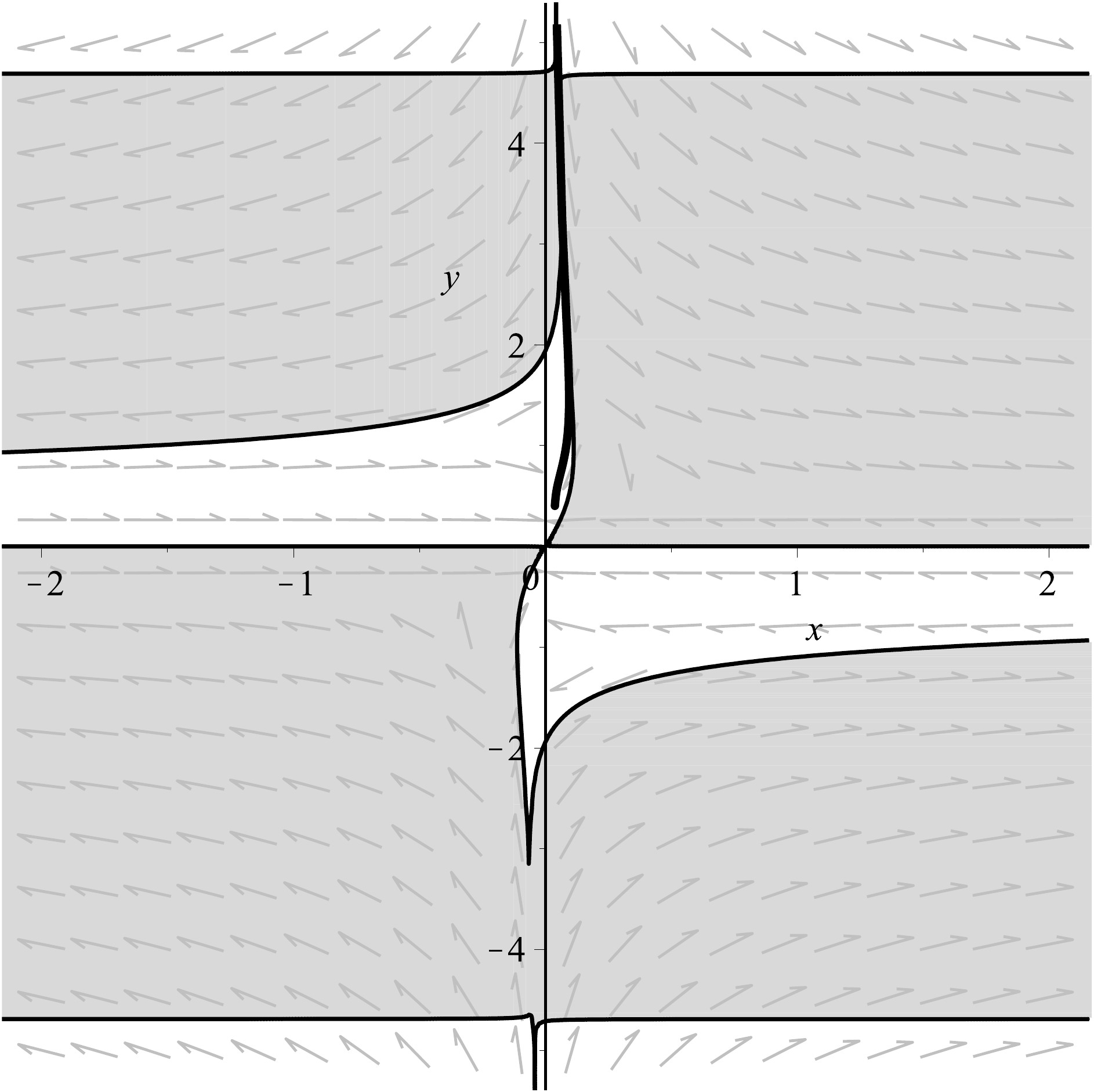} }
\subfigure[Sixth blow-up.]{
\includegraphics[width=\WID]{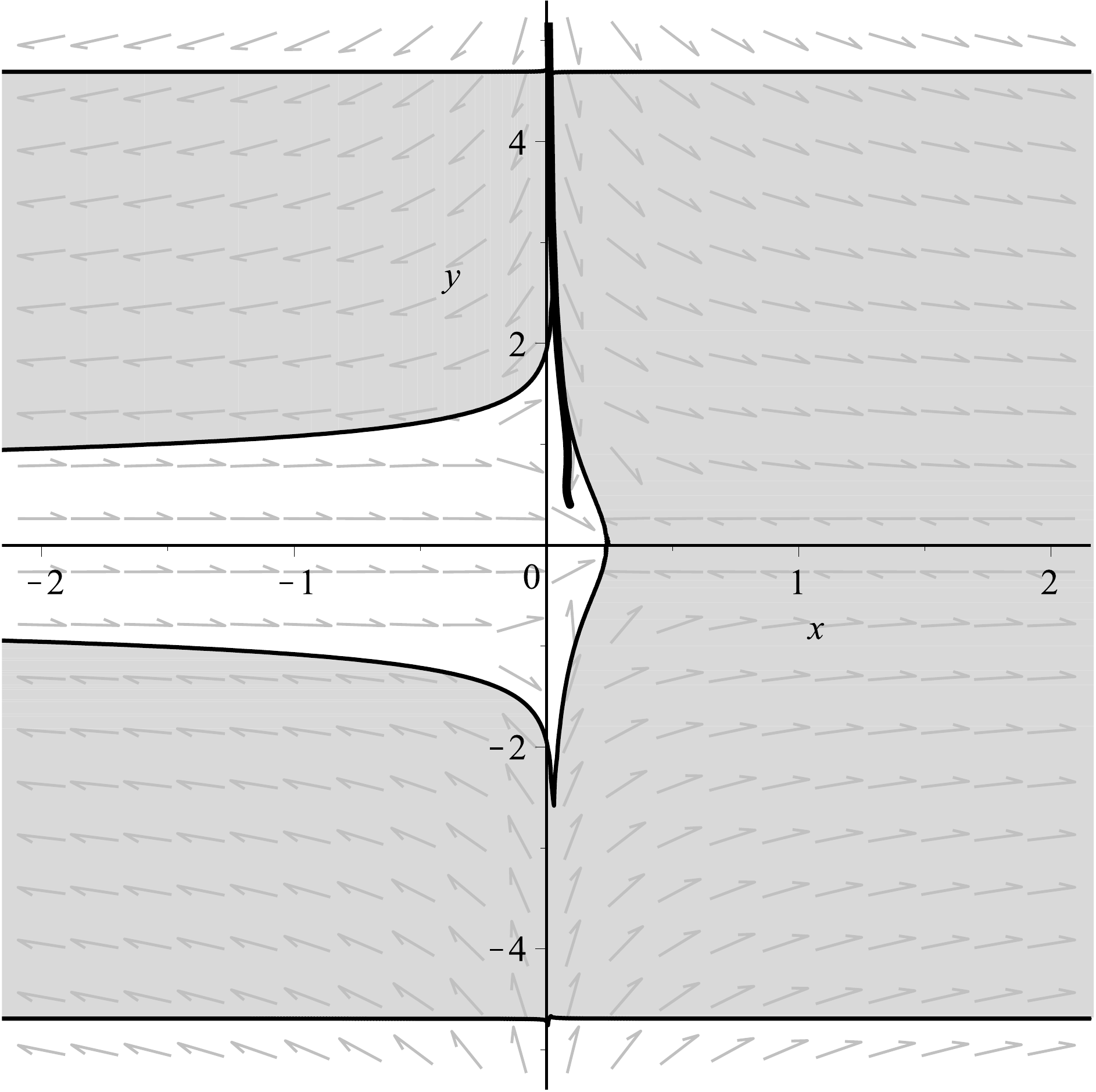} }
\subfigure[Translation.]{
\includegraphics[width=\WID]{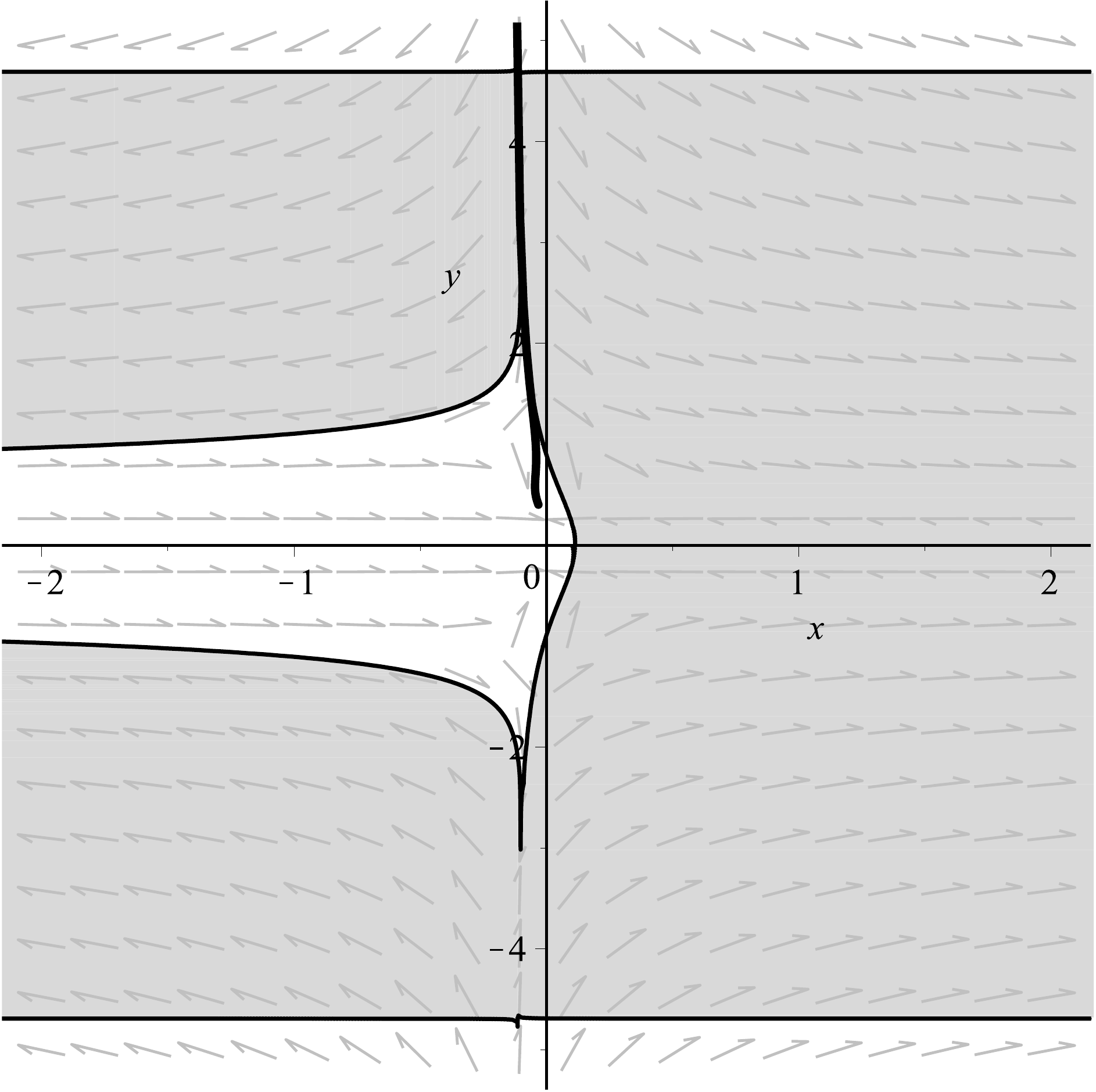}}

\caption{Analysis of the tangency at infinity of the separatrix $S$ (bold line) and the curve $\{C_t=0\}$ with $t=10$. Shaded areas
correspond to
$C_t>0$.} \label{imgblows}
\end{figure}

\newpage
\bibliographystyle{amsplain}
\bibliography{biblio}

 \end{document}